\documentclass[onefignum,onetabnum]{siamart190516}

\usepackage{lipsum}
\usepackage{amsfonts}
\usepackage{graphicx}
\usepackage{epstopdf}
\usepackage{algorithmic}

\usepackage[utf8]{inputenc}
\usepackage{amssymb}
\usepackage{amsmath}

\usepackage{graphicx}
\usepackage{epstopdf}
\usepackage{algorithmic}
\ifpdf
  \DeclareGraphicsExtensions{.eps,.pdf,.png,.jpg}
\else
  \DeclareGraphicsExtensions{.eps}
\fi

\newsiamremark{remark}{Remark}
\newsiamremark{hypothesis}{Hypothesis}
\crefname{hypothesis}{Hypothesis}{Hypotheses}
\newsiamthm{claim}{Claim}
\newsiamthm{assumptions}{Assumptions}

\graphicspath{{figures/}}
\usepackage{tikz}
\usetikzlibrary{arrows,calc}
\usepackage{siunitx}

\usepackage{stmaryrd}

\renewcommand{\vec}[1]{ \boldsymbol{#1}}

\renewcommand{\vec}[1]{ \boldsymbol{#1}}

\newcommand{\Up}{{U^{(p)}}}
\newcommand{\Uhp}{{\widehat{U}^{(p)}}}

\newcommand{\Uh}{{\widehat{U}}}
\newcommand{\tV}{{\widetilde{\mathcal{V}}}}

\newcommand{\dimVUH}{{\mathrm{dim}(\mathcal{V}_0(\Uh))}}

\usepackage{enumerate}

\newcommand{\pp}{{(p)}}

\usepackage{diagbox}
\title{Distributed solution of Laplacian eigenvalue problems 
\thanks{Submitted to the editors DATE.
\funding{
The first author was partially supported by the Stenb\"{a}ck  foundation, the second author by Magnus Ehrnrooth foundation, and the third author by the 
Academy of Finland projects with decision numbers 288980, 312340, and 324611. }}}

\author{Antti Hannukainen \thanks{Department of Mathematics and Systems Analysis, Aalto University
  (\email{antti.hannukainen@aalto.fi}, \email{jarmo.malinen@aalto.fi}, \email{antti.ojalammi@aalto.fi})}
\and Jarmo Malinen \footnotemark[2] 
\and Antti Ojalammi \footnotemark[2]}
\usepackage{amsopn}

\usepackage{amsopn}

\ifpdf \hypersetup{ pdftitle={Distributed solution of Laplacian eigenvalue problems},
  pdfauthor={A. Hannukainen, J. Malinen, and A. Ojalammi} } \fi

\headers{Distributed solution of Laplacian eigenvalue problems}{A. Hannukainen, J. Malinen, and A. Ojalammi}

\begin{document}

\maketitle

\begin{abstract} The purpose of this article is to approximately compute the eigenvalues of the symmetric Dirichlet Laplacian within an interval $(0,\Lambda)$. A novel domain decomposition Ritz method, partition of unity condensed pole interpolation method, is proposed. This method can be used in distributed computing environments where communication is expensive, e.g., in clusters running on cloud computing services or networked workstations. The Ritz space is obtained from local subspaces consistent with a decomposition of the domain into subdomains. These local subspaces are constructed independently of each other, using  data only related to the corresponding subdomain. Relative eigenvalue error is analysed. Numerical examples on a cluster of workstations validate the error analysis and the performance of the method.
\end{abstract}



\section{Introduction}
Assume that $\Omega \subset \mathbb{R}^d,d=2,3,$ is a bounded domain with Lipschitz boundary, and that $\mathcal{V} \subset H_0^1(\Omega)$ is a closed subspace. Consider the following eigenproblem: Find $(\lambda_j,u_j) \in \mathbb{R}^+ \times \mathcal{V} \setminus \{0\}$~such that 
\begin{equation} \label{eq:cont_eigen}
\int_\Omega \nabla u_j \cdot \nabla w \; dx = \lambda_j \int_\Omega u_j w \; dx \quad \quad \text{and} \quad \quad \| u_j \|_{L^2(\Omega)} = 1
\end{equation}
for each $ w \in \mathcal{V}$. Here~$\mathbb{R}^+ := (0,\infty)$, and the eigenvalues $\lambda_j$ are numbered in non-decreasing order and repeated by their multiplicities.  The purpose of this article is to compute all eigenvalues within the {\em spectral interval of interest}~$(0,\Lambda)$ for $\Lambda \in \mathbb{R}^+$  to a given accuracy in a distributed computing environment. In the following, the \emph{relevant eigenfunctions} are those that are associated to eigenvalues in $(0,\Lambda)$.

If $\mathcal{V}$ is a finite element space, it may happen that~\eqref{eq:cont_eigen}~cannot be solved using a single workstation. There are two types of distributed solution methods that can then be used. Firstly, a parallel eigenvalue iteration (such as shift-and-invert Lanczos) can be used together with a parallel solver for the shifted linear system; see, e.g., \cite{arbenz:2005}. Secondly, one can use a Domain Decomposition (DD) method such as AMLS \cite{Benninghof:AMLS:2004}, RS-DDS \cite{Saad:2018}, or the CMS variant proposed in \cite{Bourquin1992}; see also \cite{bampton:coupling_1968,saad:amls,Benninghof:AMLS:2004,hurty1960vibrations}. 

All aforementioned eigensolvers are Ritz methods. That is, instead of~\eqref{eq:cont_eigen} one solves the problem: Find $(\tilde{\lambda}_j,\tilde{u}_j)
\in \mathbb{R}^+ \times \tV \setminus \{0\}$ such that for each $ w \in \tV$
\begin{equation} \label{eq:red_eigen}
\int_\Omega \nabla \tilde{u}_j \cdot \nabla w \; dx = \tilde{\lambda}_j \int_\Omega \tilde{u}_j w \; dx \quad \quad \text{and} \quad \quad \| \tilde{u}_j \|_{L^2(\Omega)} = 1,
\end{equation}
where the {\em method subspace} $\tV \subset \mathcal{V}$ is finite-dimensional. The eigenvalues $\tilde{\lambda}_j$ are in non-decreasing order and repeated according to their multiplicities. We assume that~\eqref{eq:red_eigen} on $\tV$ can be solved exactly, and we study the relative error between the corresponding eigenvalues of~\eqref{eq:cont_eigen} and~\eqref{eq:red_eigen}. This error depends on $\tV$. The approximation error (if any) resulting from restricting the Laplacian eigenvalue problem in $H^1_0(\Omega)$ to $\mathcal{V}$ is not treated; for such error analysis in the context of finite element method, see, e.g., \cite{Boffi:2010}.

The method subspaces used in CMS, AMLS, and RD-DDS are associated to a decomposition of $\Omega$  into \emph{non-overlapping} subdomains $\{ \Omega_j \}$. They are constructed by solving two kinds of eigenproblems: small, inexpensive \emph{local problems} on each $\Omega_j$ and \emph{interface problems} related to adjacent subdomains. It is noteworthy that the interface problems are never local, and their solution accrues a significant computational cost in existing DD methods. Still, DD methods are especially useful if a large number of smallest eigenvalues is to be computed. In particular, if only a small fraction of finite element basis functions is related to the interface, AMLS can provide efficient approximation for thousands of eigenpairs.

We propose a novel DD eigensolver, {\em Partition of Unity Condensed Pole Interpolation} (PU-CPI) for the distributed solution of \eqref{eq:cont_eigen}.  PU-CPI is a Ritz method using a method subspace associated to a (finite, relatively) \emph{open cover} of $\Omega$, denoted by $\{ \Up \}$, instead of a non-overlapping decomposition. Since there are no geometric interfaces between the subdomains, solution of non-local interface problems is avoided. Consequently, only local eigenproblems, defining \emph{local subspaces} on $\Up$, have to be solved. A partition of unity on $\{\Up\}$ is used to bind these local subspaces to a conforming method subspace as in \cite{BaMe:1996}.

Because there are only local problems, PU-CPI does not require any communication between its distributed tasks associated to $\{\Up\}$. The master and workers communicate to distribute local data at the beginning, and to transfer the finished local results at the end of each task. Thus, PU-CPI can be used even if communication is expensive or nodes are not simultaneously available, e.g., on a cluster running in a cloud computing service or on networked workstations. 

We show that the eigenvalue error resulting from PU-CPI depends on how accurately the relevant eigenfunctions \eqref{eq:cont_eigen} are approximated by the local subspaces. Thus, the design of the local subspace for $\Up \subset \Omega$ requires some understanding on the behaviour of the these eigenfunctions restricted to $\Up$. It is well-known that (excluding exceptional cases) the restriction of relevant eigenfunctions to any $\Up$ can be recovered from its trace on $\partial \Up$. We exploit this property on extended subdomains $\Uhp$, $\Up \subset \Uhp \subset \Omega$, and show, intuitively speaking, that the eigenfunction restricted to $\Up$ only loosely depends on its trace on $\partial \Uhp$. Due to this loose dependency, sufficiently good local subspaces, with small dimension, can be defined without referring to boundary values of relevant eigenfunctions on $\partial \Uhp$ at all. We expect that such loose dependency is a generic property of elliptic differential operators, making our approach applicable to other problems besides \eqref{eq:cont_eigen}, e.g., linear elasticity. 

We proceed to review the major steps taken to design $\tV$. In Lemma~\ref{prop_extension} we give a representation formula that relates the restriction of a relevant eigenfunction to $U$ and its trace on $\partial \Uh$. As stated above, this restriction depends on the trace via a boundary-to-interior mapping $Z_U(\cdot)$, which is a non-linear function from $(0,\Lambda)$ to a space of bounded linear operators. We construct the local subspace for $U$ to approximate the range of this function. Lemma~\ref{lemma:compact} shows that the range of $Z_U$ consists of compact operators. We then introduce an \emph{approximate-linearise-compress} strategy in the first main results of this article,  Theorems~\ref{cor:Ziperror} and \ref{thm:trunc_err_est} to study $Z_U$. Ultimately, an estimate for the relative eigenvalue error is given in Theorem~\ref{thm:main1}.


We give a unified analysis valid both for $\mathcal{V} = H^1_0(\Omega)$ or some finite element space. Treating the continuous setting helps in choosing an appropriate inner products for subspaces needed. We envision finite element simulation as a typical application of PU-CPI. Hence, we give a detailed explanation of its application to first-order finite elements in three dimensions. 

To demonstrate the potential of PU-CPI, we compute the lowest 200 eigenvalues of~\eqref{eq:cont_eigen} where $\mathcal{V}$ is a tetrahedral first-order finite element space in $\mathbb{R}^3$. The resulting algebraic eigenvalue problem has approximately $10^7$ unknowns. The PU-CPI computation took less than two hours on a cluster of 26 networked workstations; further details are given in Section~\ref{sec:num_ex}.

In our implementation of PU-CPI, the computation proceeds in three steps:
\begin{enumerate}
\item Preparation of the computational grid by the master. 
\item \label{step:intro_distr} Distributed computation of the local subspaces by workers. Ultimately, each worker solves a local eigenvalue problem on $\Up$ leading to a dimension reduced basis.
\item Assembly and solution of the reduced eigenvalue problem by the master.  
\end{enumerate}

The article is organised as follows. We begin by reviewing the preliminaries and error analysis
of Ritz methods. In Section~\ref{sec:local-subspace}, we construct the local subspace for a single subdomain and derive the local error estimate for it. In Section~\ref{sec:pu-cpi}, we combine the local subspaces to the method subspace $\tV$ and introduce the global error estimates. Section~\ref{sec:fem} is devoted to standard first order finite element space. We conclude the article with numerical examples in Section~\ref{sec:num_ex}, followed by a discussion.

\section{Background} 

Let $\Omega^\prime, \Omega \subset \mathbb{R}^d$ for $d=2,3$ be open bounded sets with Lipschitz boundaries such that $\Omega^\prime \subset \Omega$. The inner products for $H^1(\Omega^\prime)$ and $H^1_0(\Omega^\prime)$ are 
\begin{equation*}
\begin{aligned}
(f,g)_{H^1(\Omega^\prime)} & := (\nabla f,\nabla g)_{L^2(\Omega^\prime ;\mathbb{R}^d)} + (f,g)_{L^2(\Omega^\prime)}, \\ 
(f,g)_{H^1_0(\Omega^\prime)} & := (\nabla f,\nabla g)_{L^2(\Omega^\prime;\mathbb{R}^d)}.
\end{aligned}
\end{equation*}
The corresponding norms are denoted by $\| \cdot \|_{H^1(\Omega^\prime)}$ and $\| \cdot \|_{H_0^1(\Omega^\prime)}$, respectively.

In the following, we discuss a subspace method for the eigenproblem related to the Laplace operator and its finite element discretisation, treated using the formulation in~\eqref{eq:cont_eigen} with different choices of the space $\mathcal{V}$. The  Laplace operator is treated by setting
$\mathcal{V} = H^1_0(\Omega)$, and the solution $(\lambda, u) \in \mathbb{R}^+ \times
H^1_0(\Omega) \setminus \{0 \}$ of~\eqref{eq:cont_eigen} is required to satisfy
\begin{equation}
\label{eq:eigen_vk}
-\Delta u  =  \lambda u  \quad \mbox{ in } L^2(\Omega) \quad \mbox{and} \quad \| u \|_{L^2(\Omega)} = 1.
\end{equation}
The finite element discretisation of \eqref{eq:eigen_vk} is obtained for $\mathcal{V} = \mathcal{V}_h$, where
\begin{equation}
\label{eq:fe_space1}
\mathcal{V}_h := \{ \; w \in H^1_0(\Omega) \; | \; w|_K \in P^1(K) \quad \mbox{for all} \quad K \in \mathcal{T}_h \; \}
\end{equation}
is the finite element space related to a conforming partition of $\Omega$ into simplices $\mathcal{T}_h$. Here $P^1(K)$ denotes the space of first-order polynomials on $K \subset \mathbb{R}^d$. 

We work with restrictions of functions from the space
$\mathcal{V} \subset H^1_0(\Omega)$ to a subdomain $\Omega^\prime \subset \Omega$. Denote
\begin{equation*}
  \mathcal{V}(\Omega^\prime) := \{ \; w|_{\Omega^\prime}\ \; | \; w \in \mathcal{V} \; \}
%
\quad \mbox{and} \quad
%
tr\mathcal{V}(\Omega^\prime) := \{ \; \gamma_{\partial \Omega^\prime} w \; | \; w \in \mathcal{V}(\Omega^\prime) \; \},
\end{equation*}
where $\gamma_{\partial \Omega^\prime} \in
\mathcal{B}(H^1(\Omega^\prime),H^{1/2}(\Omega^\prime))$ is the trace operator on
$H^1(\Omega^\prime)$. The space of functions with homogeneous boundary values is denoted by
\begin{equation*}
\mathcal{V}_0(\Omega^\prime) := \{ \; w \in \mathcal{V}(\Omega^\prime) \; | \; \gamma_{\partial \Omega^\prime} w = 0 \; \}.
\end{equation*}
The spaces $\mathcal{V}(\Omega^\prime)$,~$\mathcal{V}_0(\Omega^\prime)$~inherit their inner products and norms from spaces $H^1(\Omega^\prime)$, $H^1_0(\Omega^\prime)$, respectively. For $tr\mathcal{V}(\Omega^\prime)$, we use the norm
\begin{equation}
\label{eq:H1_2norm}
\| f \|_{tr \mathcal{V}(\Omega^\prime) } := \frac{1}{\sqrt{2}} \min_{\substack{w \in \mathcal{V}(\Omega^\prime) \\ \gamma_{\partial \Omega^\prime} w = f }} \| w \|_{H^1(\Omega^\prime)}.
\end{equation}
We make a standing assumption  that all these spaces are complete. This holds, e.g., if $\mathcal V = H^1_0(\Omega)$, or it is finite-dimensional. 

%

\subsection{Subspace methods} 
If~\eqref{eq:cont_eigen} is posed on $\Omega^\prime \subset \Omega$ and in a closed subspace $\mathcal{W} \subset H_0^1(\Omega^\prime)$ instead of $\mathcal{V} \subset H_0^1(\Omega)$, we denote the set of eigenvalues as $\sigma(\mathcal{W})$. 

The relative error between corresponding
eigenvalues of \eqref{eq:cont_eigen} and \eqref{eq:red_eigen} has been extensively studied, see e.g., \cite{BaOs:1989,Chatelin:1975,KnOs:07}, the review article \cite{Boffi:2010}, and the references therein. These results are not straightforward, and there exists multiple variants with different assumptions. All such bounds (that the authors are aware of) estimate the error by a product of two expressions as in \eqref{eq:bestapprox}. The latter one is related to the \emph{eigenfunction approximation error}, i.e., the accuracy of approximation of one, or several eigenfunctions of \eqref{eq:cont_eigen} in the method subspace $\tV$. The first one is an expression that may depend on $\mathcal{V}$ and $\tV$ via $\sigma(\mathcal{V})$ and $\sigma(\tV)$. If the eigenfunction approximation  error is sufficiently small, this first expression remains bounded and can be regarded as a generically unknown constant.  

In this work, an estimate adapted from \cite[Theorem 3.2]{KnOs:07}, where the relative eigenvalue error is bounded by the approximability of the corresponding eigenfunction in the method subspace $\tV$, is used since it simplifies the error estimates. Because the core of our analysis of PU-CPI is to bound the eigenfunction approximation error, the resulting bounds can be combined with other relative eigenvalue error estimates as well.  

The \emph{spectral gap} of $\mathcal V$ on $(0,\Lambda)$ is defined as
  \begin{equation}
  \label{eq:def_sgap}
     \rho_\Lambda := \min_{ \substack{\lambda,\mu \in \sigma(\mathcal V) \cap (0,\Lambda) \\ \lambda \neq \mu}} |\mu-\lambda|.
     \end{equation}
\begin{proposition}
  \label{prop:eigen_proj_estimate} 
 Let $\rho_\Lambda$ be as defined in \eqref{eq:def_sgap}, $\tV \subset \mathcal{V}$ a finite-dimensional method subspace, $1 \leq j \leq \# \{ \sigma(\mathcal V) \cap (0,\Lambda) \}$, and $(\lambda_j, u_j) \in
  \sigma(\mathcal{V}) \times  \mathcal{V} \setminus \{ 0 \}$ eigenpair of~\eqref{eq:cont_eigen} corresponding to a simple eigenvalue $\lambda_j$. Assume that the Hausdorff distance
 \begin{equation}
 \label{eq:asstV}
 \mathop{dist}\left( \sigma(\tV) \cap (0,\Lambda),\sigma(\mathcal V) \cap (0,\Lambda) \right) \leq \frac{1}{2} \rho_\Lambda.
\end{equation}
Then there exists $\tilde{\lambda} \in \sigma(\tV)$ and $C(\lambda_j) \equiv C(\lambda_j;\mathcal{V})$ such that
  \begin{equation}
  \label{eq:bestapprox} 
  \frac{| \lambda_j - \tilde{\lambda} |}{ \lambda_j} \leq
C(\lambda_j) \min_{v \in \tV} \| u_j  - v \|^2_{H^1_0(\Omega)}.
 \end{equation} 
\end{proposition}
Proposition~\ref{prop:eigen_proj_estimate} is a streamlined version of~\cite[Thm.~3.2.]{KnOs:07}. The original statement gives an explicit formula for $C(\lambda_j)$ that unfortunately depends on the (\emph{a priori} unknown) spectra $\sigma(\tV)$ and $\sigma(\mathcal{V})$. To guarantee that $C(\lambda_j)$ remains uniformly bounded independently of $\tV$, we have introduced \eqref{eq:asstV}. If the relevant eigenfunctions are sufficiently well approximated in $\tV$, the Hausdorff distance $\mathop{dist}\left( \sigma(\tV) \cap (0,\Lambda),\sigma(\mathcal V) \cap (0,\Lambda) \right)$ satisfies  \eqref{eq:asstV} by \cite[Thm.~3.1.]{KnOs:07}. Exactly when this happens in terms of $\tV$, depends on the spectral gap $\rho_{\Lambda}$ which is unknown unless the exact spectrum $\sigma(\mathcal{V}) \cap (0,\Lambda)$ is known. Hence, there is no {\it a priori} quantitative statement on \eqref{eq:asstV}. Observe that \cite{KnOs:07} uses different normalisation of eigenfunctions which affects $C(\lambda_j)$ but is later taken into account in Theorem~\ref{thm:main1}.

\subsection{The PU-CPI method subspace} \label{sec:PUM} 
Let $\{ \Up \}_{p=1}^M$ for $M\geq 2$ and $\Up \subset \Omega$, be an open cover of the domain $\Omega \subset \mathbb{R}^d$. In addition, assume that each $\Up$ has Lipschitz boundary and that there does not exists $p,q \in \{1,\ldots,M\}, p\neq q$, satisfying $\Up \subset U^{(q)}$. We proceed to describe how the PU-CPI method subspace $\tV$ is constructed from the \emph{local method subspaces} $\tV(U^\pp) \subset \mathcal{V}(U^\pp)$.


For $p=1,\ldots,M$, let the {\em stitching operators } $R^\pp \in \mathcal{B}(\mathcal{V}(\Up),\mathcal{V})$ satisfy
\begin{equation*}
(R^\pp w^\pp)|_{\Omega \setminus \Up} = 0 \quad \mbox{and} \quad
\sum_{p=1}^M R^\pp (w|_\Up) = w \quad 
\end{equation*}
for each $w^\pp \in \mathcal{V}(\Up)$ and $w \in \mathcal V$. Suitable operators $\{ R^\pp \}_{p=1}^M$ for $\mathcal{V} = H^1_0(\Omega)$ can be obtained by multiplication with a partition of unity associated to $\{\Up\}$ as in \cite{BaMe:1996}. The PU-CPI method subspace $\tV$, depending on the local method subspaces $\{ \tV(U^\pp) \}_{p=1}^M$, is defined as
\begin{equation}
\label{eq:PUMdef}
\tV := \left \{ w \in \mathcal{V} \; | \; w = \sum_{p=1}^M R^\pp w^\pp \quad \mbox{for} \quad w^\pp \in \tV(\Up) \; \right \}
\end{equation}
where each $\tV(\Up) \subset \mathcal{V}(\Up)$ has a low dimension. If $\tV$ satisfies \eqref{eq:PUMdef} and the assumptions of Proposition~\ref{prop:eigen_proj_estimate}, the  eigenvalue error depends on approximation properties of the local subspaces. Using a similar technique as in \cite{BaMe:1996} gives 
\begin{equation}
  \label{eq:pum_error}
\min_{v \in \tV} \| u_j  - v \|_{H^1_0(\Omega)} \leq 
\| G \|_{L^\infty(\Omega)} \left( \sum_{p=1}^M  \mathcal{E}(u_j,\Up) \right)^{1/2},
\end{equation}
where function $\mathcal{E}$ is the {\em local approximation error},
\begin{equation}
\label{eq:loc_app_error}
\mathcal{E}(u,\Up) := \min_{ w \in \tV(\Up)} \int_{\Up} \left| \nabla \left[ R^\pp (u|_{\Up}-w)\right] \right|^2 \; dx 
\end{equation}
and $G:\Omega \rightarrow \{1,\ldots,M\}$ is defined as $G(x) := \#\{ p \;|\; x \in \Up \}$. The aim is to design the local method subspaces $\tV(\Up)$ so that both  $\dim(\tV(\Up))$ and $\mathcal{E}_{j,p} \equiv \mathcal{E}(u_j,\Up)$  are small.

\section{Local method subspace}
\label{sec:local-subspace}
A local method subspace ${\tV(\Up) \subset \mathcal{V}(\Up)}$ for a single subdomain $\Up$ is designed next. For notational convenience, denote $U = \Up$, $R = R^\pp$, and let $(\lambda,u)$ be some solution to \eqref{eq:cont_eigen} satisfying
$\lambda < \Lambda$.


\subsection{Extended subdomain}
\label{sec:extended_subdomain}
Given $r>0$ and $U \subset \Omega$, let $\Uh \subset \Omega$ be a domain satisfying
\begin{equation}
\label{eq:kaulus}
\{ \; x \in \Omega \; | \; dist(x,U) < r \} \subset \Uh.
\end{equation}
Any such $\Uh$ is called an {\em $r$-extension} of $U$, and we make it a standing assumption that both $U$ and $\Uh$ have Lipschitz boundaries. By our assumptions, $U \neq \Omega$, and hence $U \neq \Uh$. In the following, $\Uh$ is fixed unless otherwise stated. The effect of the parameter $r$ is numerically studied in Section~\ref{sec:num_ex}. As shown in the next section, the essential component of the PU-CPI method is the operator-valued function $Z_U : (0,\Lambda) \to \mathcal{B}(tr\mathcal{V}(\Uh),\mathcal{V}(U))$. It will be shown that $Z_U$ is, in fact, analytic, and due to the use of the $r$-extension and elliptic regularity also compact operator-valued.

\subsection{Eigenfunction representation formula}\label{sec:eigenf-repr-form}
We represent $u|_\Uh$ in terms of its boundary trace $\gamma_{\partial \Uh} \left(u|_\Uh \right)$. By \eqref{eq:cont_eigen},
$u|_\Uh$ satisfies
\begin{equation}
\label{eq:CPIstart}
\int_\Uh \left( \nabla u|_\Uh \cdot \nabla w - \lambda u|_\Uh w \right)\; dx = 0  
\end{equation}
for each $w \in \mathcal{V}_0(\Uh)$. We assume that there exists a right inverse  $E \in \mathcal{B}(tr\mathcal{V}(\Uh),\mathcal{V}(\Uh) )$ of $\gamma_{\partial \Uh}$ satisfying
\begin{equation}
\label{eq:Edef}
  E: tr\mathcal{V}(\Uh) \rightarrow \{ \; v \in \mathcal{V}(\Uh) \; | \; v|_U = 0 \; \}.
\end{equation}
Such $E$ always exists if $\mathcal V = H^1_0(\Omega)$. If the finite element method is used for defining  $\mathcal{V}$, $\{ \Up\}$ and $\{\Uhp\}$ are \emph{constructed} so that $E$ exists, see Section~\ref{sec:fem}.  

Equation \eqref{eq:CPIstart} is solved by decomposing
\begin{equation}
\label{eq:udeco}
u|_{\Uh} = u_0 + Eu_B \quad \mbox{where} \quad u_0 \in \mathcal{V}_0(\Uh) \quad \mbox{and} \quad u_B := \gamma_{\partial \Uh} \left( u|_\Uh \right).
\end{equation}
It follows from~\eqref{eq:Edef} that $u|_{U} = u_0|_U$. Using the decomposition in \eqref{eq:udeco}, \eqref{eq:CPIstart} gives
\begin{equation}
\label{eq:u_0}
\int_\Uh \left( \nabla u_0 \cdot \nabla w - \lambda u_0 w \right) \; dx = -  \int_\Uh \left( \nabla w \cdot \nabla E u_B - \lambda w Eu_B  \right) \; dx 
\end{equation}
for each $w \in \mathcal{V}_0(\Uh)$, which defines $u_0$ as a function of $\lambda$ and $u_B$. We proceed as in~\cite{CPI:2018} and use an $L^2(\Uh)$-orthonormal eigenbasis expansion to solve \eqref{eq:u_0}. Let $(\mu_k, v_k) \in \mathbb{R}^+ \times \mathcal{V}_0(\Uh) \setminus \{0 \}$ be such that 
\begin{equation}
\label{eq:eigen_vk2}
\int_{\Uh} \nabla v_k \cdot \nabla w  \; dx  =  \mu_k \int_{\Uh} v_k w \; dx  \quad \mbox{and} \quad \| v_k \|_{L^2(\Uh)} = 1
\end{equation}
for each $w \in\mathcal{V}_0(\Uh)$. Assume that $\{ \mu_k \}_k \subset \mathbb{R}^+$ are indexed in non-decreasing order and repeated
according to their multiplicities. The set $\{ v_k \}_k$ is $L^2(\Uh)$-orthonormal in  $\mathcal{V}_0(\Uh)$,
 hence $\{v_k/\sqrt{\mu_k} \}_k$ is an $H_0^1(\Uh)$-orthonormal basis of $\mathcal{V}_0(\Uh)$.
 %
%
%
To solve $u_0$ from \eqref{eq:u_0}, expand in $H^1_0(\Uh)$
\begin{equation}
\label{eq:v_expansion}
u_0 = \sum_{j=1}^\dimVUH \alpha_j v_j \quad \mbox{where each} \quad \alpha_j \in \mathbb{R}.
\end{equation}
Using this expansion with~\eqref{eq:u_0}~and setting $w=v_k$ in~\eqref{eq:eigen_vk2}, the
orthogonality of the eigenfunctions gives
 \begin{equation}
 \label{eq:ZUaputulos1}
   \alpha_k(\mu_k - \lambda) = -\int_\Uh \left( \nabla v_k \cdot \nabla Eu_B- \lambda  v_k Eu_B \right) \;dx.
 \end{equation}
If $\lambda \not \in \sigma(\mathcal{V}_0(\Uh))$, $u_0$ is determined by solving $\alpha_k$ for $k = 1,\ldots,
\dim \mathcal{V}_0(\Uh)$. To treat any $\lambda
\in (0,\Lambda)$, we split the coefficients $\alpha_k$ into two groups using the parameter $\tilde{\Lambda} > \Lambda$ and $K:\mathbb{R}^+ \to \mathbb{N}$, given by
\begin{equation*}
K(t) := \# \{ \; \mu_k \in \sigma(\mathcal{V}_0(\Uh)) \; | \; \mu_k \leq t \; \}.
\end{equation*}
Since $\lambda \in (0,\Lambda)$ the coefficients $\alpha_k$ in~\eqref{eq:v_expansion} for $k > K(\tilde{\Lambda})$ are obtained from \eqref{eq:ZUaputulos1}. We have now proved the following lemma:
\begin{lemma} \label{prop_extension} Let $\tilde{\Lambda} > \Lambda > 0$. Assume that $(\lambda,u) \in (0,\Lambda) \times \mathcal{V}$ and $U\subset \Uh \subset \Omega$, $\Uh \neq \Omega$, satisfy~\eqref{eq:cont_eigen}~and~\eqref{eq:kaulus}, respectively. Then we have the following orthogonal splitting in $L^2(\Uh)$ and in $H^1_0(\Uh)$:
\begin{equation}
\label{eq:rep_uU}
u|_U =  \sum_{k=1}^{K(\tilde{\Lambda})} \alpha_k \, v_k|_U 
+
( Z(\lambda) u_B )|_U,
\end{equation}
where $u_B = \gamma_{\partial \Uh} u$, $\{ \alpha_k \}_{k=1}^{K(\tilde{\Lambda})} \subset \mathbb{R}$, and $Z:(0,\Lambda) \rightarrow \mathcal{B}(tr\mathcal{V}(\Uh), \mathcal{V}_0(\Uh))$ is defined as
\begin{equation}
\label{eq:fun_z}
Z(t) w_B :=  \sum_{k=K(\tilde{\Lambda})+1}^\dimVUH
\frac{ v_k}{\mu_k - t} \int_{\Uh} (-\nabla v_k \cdot \nabla Ew_B + t v_k Ew_B) \; dx.
\end{equation}
The sum converges uniformly for $t\in (0,\Lambda)$ in $H^1_0(\Uh)$ and $L^2(\Uh)$. Moreover, $Z(t)$ is analytic function for $t\in (0,\Lambda)$. 
\end{lemma}
There are many ways of showing that $Z(t) \in \mathcal{B}(tr\mathcal{V}(\Uh), \mathcal{V}_0(\Uh))$ for $t\in (0,\Lambda)$;~e.g., by using Lemma~\ref{lemma:c1c2}. The function $Z$ depends implicitly on $\tilde{\Lambda}, \Uh$ and $U$ in addition to $t$. 



\subsection{Evaluation of $Z$}\label{sec:evaluation-z}
Following the approach used in \cite{CPI:2018}, we discuss how $Z$ can be evaluated given $K(\tilde{\Lambda})$ lowest eigenmodes\footnote{If $\mathcal{V}$ is a finite element space, the value of $K(\tilde{\Lambda})$ can be computed using $LDL^T$-decomposition and Sylvester's law of inertia. These kinds of decompositions are computed internally in eigensolvers, and we consider evaluating $K(\tilde{\Lambda})$ as an implementation issue.} of~\eqref{eq:eigen_vk2}. Denote
\begin{equation*}
E_{\tilde{\Lambda}} := \mathop{span} \{ v_1,\ldots,v_{K(\tilde{\Lambda})} \} \quad \mbox{for $(v_k,\mu_k)$ satisfying \eqref{eq:eigen_vk2}}.
\end{equation*}
%
Fix $t \in (0,\Lambda),\; w_B \in tr \mathcal{V}(\Uh)$, and solve the
auxiliary problem: Find $\hat{z}_0(t) \in \mathcal{V}_0(\Uh)$ such that
\begin{equation}
  \label{eq:z_sol}
\int_{\Uh} 
\left( \nabla \hat{z}_0(t) \cdot \nabla w - t \hat{z}_0(t) w \right) \; dx = 
- \int_{\Uh} \left( \nabla w \cdot \nabla E w_B  - t w Ew_B  \right)\; dx 
\end{equation}
for each $w \in \mathcal{V}_0(\Uh)$. As in Section~\ref{sec:eigenf-repr-form}, each solution admits the orthogonal splitting
\begin{equation*}
  \hat{z}_0(t) = \sum_{k=1}^{K(\tilde{\Lambda})} \alpha_k v_k + Z(t)w_B \in E_{\tilde{\Lambda}} \oplus E_{\tilde{\Lambda}}^\perp
\end{equation*}
\noindent even though some $\alpha_k$'s cannot be \emph{uniquely} solved from~\eqref{eq:z_sol} for the exceptional $t \in \sigma(\mathcal{V}_0(\Uh))$. After $\hat{z}_0(t)$ has been solved from~\eqref{eq:z_sol}, $Z(t)w_B$ can be evaluated as $Z(t)w_B = P \hat{z}_0(t)$, where $P \in \mathcal{B}(\mathcal{V}_0(\Uh))$ is the $L^2(\Uh)$ orthogonal projection onto $E^\perp_{\tilde \Lambda}$.
%
%
\subsection{The complementing subspace} 
\label{sec:complementing_subspace}
Our aim is to design the finite-dimensio\-nal subspace $\tV(U)$ such that the local approximation error in \eqref{eq:loc_app_error}, namely
\begin{equation*}
\min_{v \in \tV(U)} \int_{U} |\nabla [ R (u|_U - v) ] |^2 \; dx,
\end{equation*}
can be made arbitrarily small
for any $(\lambda,u) \in  (0,\Lambda) \times \mathcal{V}$ satisfying  \eqref{eq:cont_eigen}. 
For $w_B \in tr \mathcal{V}(\Uh)$ and $t \in (0,\Lambda)$, denote
\begin{equation}
\label{eq:ZUdef}
    Z_U(t) w_B = (Z(t) w_B)|_U.
\end{equation}
Obviously by Lemma~\ref{prop_extension} and boundedness of the restriction operator, we have  $Z_U: (0,\Lambda)\to\mathcal{B}(tr\mathcal{V}(\Uh),\mathcal{V}(U))$. By Lemma~\ref{prop_extension}, 
\begin{equation}
\label{eq:Uexp} 
u|_U =  \sum_{k=1}^{K(\tilde{\Lambda})} \alpha_k \, v_k|_U 
+
Z_U(\lambda) u_B,
\end{equation}
for some real-valued $\alpha_k$'s. We construct $\tV(U)$ according to the splitting in \eqref{eq:Uexp} as
\begin{equation}
\label{eq:VU_def}
\tV(U) = E_{\tilde{\Lambda}}(U) \oplus \mathcal{W}(U) \quad \mbox{where} \quad E_{\tilde{\Lambda}}(U) := \mathrm{span} \{ v_1|_U,\ldots,v_{K(\tilde{\Lambda})}|_U \},
\end{equation}
and $\oplus$ denotes the orthogonal direct sum in $\mathcal{V}$. The space $\mathcal{W}(U)$ is called the {\em local complementing subspace}. Let 
\begin{equation}
\label{eq:WUerror}
e_U(\mathcal{W}(U)) :=  \sup_{\substack{t\in (0,\Lambda) \\ w \in \mathcal{V}{(\Uh)}}} \inf_{v \in \mathcal{W}(U)} \frac{\int_U |\nabla [ R(Z_U(t) w_B - v) ] |^2 \; dx}{\| w \|^2_{H^1(\Uh)}},
\end{equation}
where $w_B = \gamma_{\partial\Uh} w$. As the first term on the right hand side of \eqref{eq:Uexp} is included in  $\tV(U)$, the local approximation error of $u$ on $U$ has the estimate
\begin{equation}
\label{eq:local_approximation_error_estimate}
\mathcal{E}(u,U) \equiv \min_{v \in \tV(U)} \int_{U} |\nabla [ R (u|_U - v) ] |^2 \; dx \leq e_U(\mathcal{W}(U)) \| u|_\Uh \|^2_{H^1(\Uh)},
\end{equation}
for each $(\lambda,u) \in  (0,\Lambda) \times \mathcal{V}$ satisfying
\eqref{eq:cont_eigen}. 

Next, we design the local complementing subspace $\mathcal{W}(U)$ so that the local approximation error in \eqref{eq:loc_app_error} can be made arbitrarily small. We begin with the {\em interpolation step}. Denote the a set of $N \geq 1$ Chebyshev nodes on the interval $(0,\Lambda)$ as $\{ \xi_i \}_{i=1}^N \subset (0,\Lambda)$. Define the interpolant $\hat{Z} : (0,\Lambda) \to \mathcal{B}( tr \mathcal{V}(\Uh),\mathcal{V}_0(\Uh) )$ as
\begin{equation}
\label{eq:Zapprox}
\hat{Z}(t) =  \sum_{i=1}^N \ell_i(t) Z(\xi_i) 
\quad \mbox{where} \quad 
\ell_i(t) = \prod_{\substack{1 \leq j \leq N \\ j \neq i}} \frac{t - \xi_j}{\xi_i - \xi_j} \quad \textrm{ for } i=1,\ldots,N
\end{equation}
are the Lagrange interpolation polynomials. The interpolation error $\widehat{Z}-Z$ is studied in Section~\ref{sec:ip_error}. We proceed with a {\em linearisation step}. Define a linear operator\footnote{Here $tr\mathcal{V}(\hat{U}; \mathbb{R}^N)$ is defined as $[tr\mathcal{V}(\hat{U})]^N$ and equipped with the natural Hilbert space norm.}
\begin{equation}
\label{eq:Bdef}
B \in \mathcal{B}(tr\mathcal{V}(\Uh ; \mathbb{R}^N),\mathcal{V}(U)) \quad \mbox{as} \quad  B \vec{v}_B := \begin{bmatrix} Z_U(\xi_1) & \ldots & Z_U(\xi_N) \end{bmatrix} \vec{v}_B.
\end{equation}
Here $\hat{Z}_U(t)w_B = (\hat{Z}(t)w_B)|_U$ for all $w_B \in tr \mathcal{V}(\Uh)$ and $t\in(0,\Lambda)$. Furthermore,
\begin{equation*}
B \vec{\ell}(t) w_{B} = B \begin{bmatrix} \ell_1(t)w_{B} & \ldots & \ell_N(t)w_{B} \end{bmatrix}^T = \hat Z_U(t) w_B,
\end{equation*}
and, hence, $\mathrm{range}(\hat{Z}_U(t)) \subset \mathrm{range}(B)$ for any $t \in (0,\Lambda)$. 

We continue with the {\em finite-rank approximation step}. Given the finite-rank operator $\widehat{B} \in \mathcal{B}( tr \mathcal{V}(\Uh;\mathbb{R}^N), \mathcal{V}(U))$, the complementing subspace is fixed as $\mathcal{W}(U) := \mathrm{range}(\widehat{B})$.  We will show that $e_U(\mathcal{W}(U))$ in \eqref{eq:WUerror} is bounded from above by
\begin{equation*}
\| R (Z_U(t) - \hat{Z}_U(t)) \|_{\mathcal{B}( tr(\mathcal{V}(\Uh),\mathcal{V}) } \quad \mbox{and} \quad \| B - \widehat{B} \|_*,
\end{equation*}
resulting in Theorem~\ref{thm:main1}. 

If $\mathcal{V} = H^1_0(\Omega)$ each of the operators $Z_U(\xi_i), i=1,\ldots,N$, is compact, which makes finding $\widehat{B}$ feasible:
\begin{lemma}
\label{lemma:compact} Let $U \subset \Uh \subset \Omega \subset \mathbb{R}^d$, $\Uh \neq \Omega$, be as in \eqref{eq:kaulus}, $\mathcal{V} = H^1_0(\Omega)$, and $Z_U: (0,\Lambda) \to
\mathcal{B}( tr \mathcal{V}(\Uh;\mathbb{R}^N), \mathcal{V}(U))$ be as defined in~\eqref{eq:ZUdef}. In addition, assume that $\Uh$ is a convex polygonal ($d=2$) or convex polyhedral domain ($d=3$). Then $Z_U(t)$ is a compact operator from $tr \mathcal{V}(\Uh)$ to $\mathcal{V}(U)$ for all $t\in(0,\Lambda) \setminus \sigma(\mathcal{V}_0(\Uh))$. 
\end{lemma}
\noindent This lemma is proved below. 

Representing $u|_U$ in terms of $\gamma_{\partial \Uh} u$ is motivated by Lemma~\ref{lemma:compact}, keeping in mind that compact operators can be approximated by finite-rank operators in operator norm. Further, the same holds for $B$ in~\eqref{eq:Bdef}~since the number $N$ of Chebyshev nodes is finite. We need the following proposition:
\begin{proposition}
\label{prop:closedgraph}
Let $\mathcal{U},\mathcal{X},\mathcal{Y}$ be Banach spaces, $T \in \mathcal{B}(\mathcal{U},\mathcal{Y})$, $\mathop{range}(T) \subset \mathcal{X}$, and $\mathcal{X}$ continuously embedded in $\mathcal{Y}$. Then $T \in \mathcal{B}(\mathcal{U},\mathcal{X})$. In addition, if the embedding $\mathcal{X} \subset \mathcal{Y}$ is compact, then $T$ is a compact operator from $\mathcal{U}$ to $\mathcal{Y}$. 
\end{proposition} 
\begin{proof} Let $u_j \rightarrow u$ in $\mathcal U$ and $Tu_j \rightarrow x$ in $\mathcal X$. Since $T : \mathcal U \rightarrow \mathcal Y$ is bounded, $T u_j \rightarrow T u$ in $\mathcal Y$. As $\mathcal X$ is continuously embedded in $\mathcal{Y}$, $x=Tu$ as equality in $\mathcal{X}$.  We have now shown that $T : \mathcal U \rightarrow \mathcal X$ is a closed linear operator. The first claim follows from the closed graph theorem. The second claim follows since the composition of a compact operator and a bounded operator is compact. 
\end{proof}
Hence, if $\mathcal{V} = H_0^1(\Omega)$ and $t\in(0,\Lambda)$, the compactness of $Z_U(t) \in \mathcal{B}( tr \mathcal{V}(\Uh;\mathbb{R}^N), \mathcal{V}(U))$ follows by showing that $Z_U(t) w_B \in H^2(U)$ for all $t \in (0,\Lambda)$ and $w_B \in tr \mathcal V(\Uh)$. Due to standing assumptions made on $U$, $H^2(U)$ is compactly embedded in $H^1(U)$; see, e.g.,~\cite[Th. 6.3]{Adams:Sobolev:1975}. 
\begin{proposition}
\label{prop:regularity}
Let the domains $U \subset \hat{U} \subset \Omega \subset \mathbb{R}^d$ be as in \eqref{eq:kaulus}. Let $\mathcal{V}=H^1_0(\Omega)$, and $u \in \mathcal{V}(\Uh)$ such that $\Delta u \in L^2(\Uh)$. Assume that one of the following holds:
\begin{enumerate}
\item[\textrm{(i)}]  $\partial \Omega \cap \partial U  = \emptyset$;
\end{enumerate}
\begin{enumerate}
\item[\textrm{(ii)}] $d=2$, $\Uh$ is a convex polygonal domain, and $\partial \Omega \cap \partial U \neq \emptyset$; or
\item[\textrm{(iii)}] $d=3$, $\Uh$ is a convex polyhedral domain, and $ \partial \Omega \cap \partial U \neq \emptyset$.
\end{enumerate}
Then $u|_U \in H^{2}(U)$.
\end{proposition}
\noindent Cases \textrm{(i)} and \textrm{(ii)} are illustrated in Figure~\ref{fig:my_label}. 
\begin{proof} If (i) holds, the claim follows from the interior regularity estimate; see, e.g.,~\cite[Ch 6.3]{evans:1998}. Assume that (ii) or (iii) holds. Let $\varphi \in C^\infty(\Uh)$ be a cut-off function satisfying $\varphi = 1 $ in $U$ and $\varphi = 0$ on $\partial \Uh \setminus \partial \Omega$. The function $\varphi u \in H^1_0(\Uh)$ satisfies $\Delta(\varphi u) \in L^2(\Uh)$ by a straightforward computation. By \cite[Ch. 2.4 \& 2.6]{grisvard:1992singularities} and assumptions (ii), (iii), we have $\varphi u\in H^{2}(\Uh)$. The claim follows from $(\varphi u)|_U = u|_U$.
\end{proof}
\begin{figure}
    \centering
    \includegraphics[width=.9\textwidth]{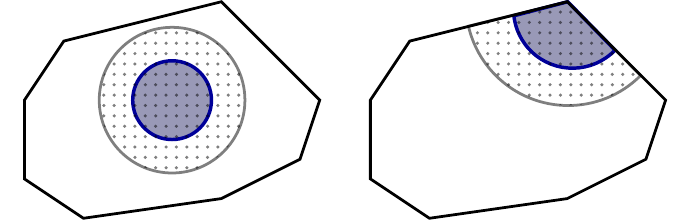}
    \caption{Illustration of the cases \textrm{(i)} and (\textrm{ii}) in Proposition \ref{prop:regularity} for $\Omega \subset \mathbb{R}^2$. The solid black line depicts $\partial \Omega$, and the subdomains $U$, $\Uh$ are shown by the blue and the grey dotted areas, respectively. There are two exclusive cases, namely (\textrm{i}) $\partial \Omega \cap \partial U = \emptyset$ (left) and (\textrm{ii}) $\partial \Omega \cap \partial U \neq \emptyset$ (right).}
    \label{fig:my_label}
\end{figure}
We complete this section by giving a proof of Lemma~\ref{lemma:compact}. 
\begin{proof}[Proof of Lemma~\ref{lemma:compact}]
Fix $t \in (0,\Lambda) \setminus \sigma(\mathcal{V}_0(\Uh))$ and $w_B \in tr \mathcal V (\Uh)$. Let $\hat{z} \in \mathcal{V}(\Uh)$ be the variational solution of 
\begin{equation*}
  \begin{cases}
    \begin{alignedat}{4}
      (\Delta + t) \hat{z} &= 0 \quad &&\mbox{in} \quad &&\Uh, \\
      \hat{z} &= w_B \quad &&\mbox{on} \quad \partial &&\Uh,
    \end{alignedat}
  \end{cases}
\end{equation*}
obviously satisfying $\hat{z} \in H^1(\Uh)$ and $\Delta \hat{z} \in L^2(\Uh)$. Similar to Section~\ref{sec:eigenf-repr-form}, decompose $\hat{z} = \hat{z}_0 + Ew_B$, where $E$ satisfies \eqref{eq:Edef}. As $t \not \in \sigma(\mathcal{V}_0(\Uh))$, 
\begin{equation*}
\hat{z}_0 = Z(t)w_B + \sum_{k=1}^{K(\tilde{\Lambda})} \frac{v_k}{\mu_k-t} \int_{\Uh} \left( -\nabla v_k \cdot \nabla E w_B + t v_k E w_B \right) \; dx.
\end{equation*}
Further, using $\hat{z} = \hat{z}_0 + Ew_B$ gives 
\begin{equation*}
Z(t)w_B + E w_B = \hat{z} - \sum_{k=1}^{K(\tilde{\Lambda})} \frac{v_k}{\mu_k-t} \int_{\Uh} \left( -\nabla v_k \cdot \nabla E w_B + t v_k E w_B \right) \; dx.
\end{equation*}
Since the sum on the right hand side has a finite number of terms where $\Delta v_k \in L^2(\Uh)$, it follows that $\Delta \left(Z(t)w_B + E w_B\right) \in L^2(\Uh)$. Using Proposition~\ref{prop:regularity} and the property $(Ew_B)|_U = 0$ gives $\left(Z(t)w_B + E w_B\right)|_U = Z_U(t)w_B \in H^2(U)$.  Since it is already known that $Z_U(t) \in \mathcal{B}(tr\mathcal{V}(\Uh),\mathcal{V}(U))$, proposition~\ref{prop:closedgraph} completes the proof.
\end{proof}

%
%
%

%


%
%
%
%
\begin{remark} The assumption of convexity for the compactness of $Z_U(t)$ can be relaxed using a more technical variant of Proposition~\ref{prop:regularity} stated in weighted Sobolev spaces, see e.g.,~\cite{nicaise:1997}. In addition, Lemma~\ref{lemma:compact} can be extended to cover all values $t\in(0,\Lambda)$.
\end{remark}
\subsection{Interpolation error}
\label{sec:ip_error}
Next, we study how the error terms
\begin{equation}
\label{eq:uerror}
e_0 := \left\| \left(\hat{Z}(t) - Z(t)\right)w_B  \right\|_{L^2(\Uh)} \quad \mbox{and} \quad  e_1 := \left\| \left(\hat{Z}(t)- Z(t) \right)w_B \right\|_{H^1_0(\Uh)},
\end{equation}
depend on $N$ and $\tilde{\Lambda}$. By Lemma~\ref{prop_extension}, the function $Z(t) w_B \in \mathcal{V}_0(\Uh)$ admits the expansion
\begin{equation*}
Z(t) w_B = \sum_{k=K(\tilde{\Lambda})+1}^\dimVUH
\frac{c_{1,k}(w_B) + t c_{0,k}(w_B) }{\mu_k - t} v_k
\end{equation*}
where the coefficients $c_{0,k} : tr\mathcal{V}(\Uh) \rightarrow \mathbb{R}$ and  $c_{1,k} : tr\mathcal{V}(\Uh) \rightarrow \mathbb{R}$ are defined as
\begin{equation}
\label{eq:cof_cdef}
c_{0,k}(w_B) := (v_k, Ew_B)_{L^2(\Uh)} \quad \mbox{and} \quad
c_{1,k}(w_B) := -(\nabla v_k, \nabla Ew_B)_{L^2(\Uh;\mathbb{R}^d)}
\end{equation}
for $k=1,\ldots,\dimVUH$. A technical estimate related to these coefficients is given in the following lemma.
\begin{lemma}
  \label{lemma:c1c2}
Let $c_{0,k}(w_B)$ and $c_{1,k}(w_B)$ be as in~\eqref{eq:cof_cdef}.
Then
%
\begin{equation*}
\sum_{k=K(\tilde{\Lambda})+1}^\dimVUH c^2_{0,k}(w_B) \leq \| E w_B \|^2_{L^2(\Uh)}  \quad \textrm{and} \quad  
\sum_{k=K(\tilde{\Lambda})+1}^\dimVUH \frac{ c_{1,k}^2(w_B) }{\mu_k} \leq \| \nabla (E w_B) \|_{L^2(\Uh;\mathbb{R}^d)}^2 .
\end{equation*}
\end{lemma}
\begin{proof}
We only prove the latter inequality. For any $u \in \mathcal{V}(\Uh)$
define $P_1 u \in \mathcal{V}_0(\Uh)$ uniquely by
the Riesz representation theorem on the Hilbert space
$\mathcal{V}_0(\Uh)$, requiring
\begin{equation*}
(\nabla P_1 u,\nabla v)_{L^2(\Uh;\mathbb{R}^d)} = (\nabla u,\nabla
v)_{L^2(\Uh;\mathbb{R}^d)} \text{ for each } v \in
\mathcal{V}_0(\Uh).
\end{equation*}
Then the mapping $u \mapsto P_1 u$ is linear, it satisfies $P_1^2 = P_1$, and $\| \nabla (P_1 E w_B) \|_{L^2(\Uh;\mathbb{R}^d)} \leq \| \nabla E w_B \|_{L^2(\Uh;\mathbb{R}^d)}$.  Since
$P_1 u$ is uniquely defined, it also follows that $P_1 u = u$ for all
$u \in \mathcal{V}_0(\Uh)$.  Hence, $P_1$ is a projection
on $\mathcal{V}(\Uh)$ with $\mathrm{range}(P_1)=\mathcal{V}_0(\Uh)$.  Since $\{
v_k/\sqrt{\mu_k} \}_{k}$ is orthonormal basis of $\mathcal{V}_0(\Uh)$, we have
\begin{equation*}
P_1 E w_B
= -\sum_{k=1}^\dimVUH \frac{v_k}{\sqrt{\mu_k}} \cdot \frac{c_{1,k}(w_B)}{\sqrt{\mu_k}}.
\end{equation*}
The claim follows using Parseval's identity.
\end{proof}

%
%
Denote 
\begin{equation*}
f_{m,k}(t) = t^{1-m} (\mu_k - t)^{-1} 
\end{equation*}
for $m=0,1$ and $k = K(\tilde{\Lambda})+1,\ldots,\dimVUH$. Recalling \eqref{eq:Zapprox}, we have
\begin{equation}
\label{eq:Ziperror} 
\begin{aligned}
\left(Z(t) - \hat{Z}(t) \right)w_B =  & \sum_{k=K(\tilde{\Lambda})+1}^\dimVUH
\left( f_{1,k}(t) - \sum_{i=1}^N  \ell_i(t) f_{1,k}(\xi_i) \right) c_{1,k}(w_B) v_k  \\ &
+ \sum_{k=K(\tilde{\Lambda})+1}^\dimVUH
\left( f_{0,k}(t) - \sum_{i=1}^N  \ell_i(t) f_{0,k}(\xi_i) \right) c_{0,k}(w_B) v_k.  
\end{aligned}
\end{equation}
Observe that the expressions in parentheses in \eqref{eq:Ziperror} are Lagrange interpolation errors with Chebyshev nodes $\{ \xi_i \}_{i=1}^N \subset (0,\Lambda)$. The derivatives of $f_{m,k}$ satisfy 
\begin{equation}
\label{eq:Df}
\frac{1}{N!} \frac{\mathrm{d}^N f_{m,k} }{\mathrm{d}t^N} (t) = \frac{\mu_k^{1-m}}{(\mu_k-t)^{(N+1)}}
\end{equation}
Hence, we have the estimate for $k = K(\tilde{\Lambda})+1,\ldots,\dimVUH$
\begin{equation}
\label{eq:Lagrange_ip}
\mu_k^{1+m} \| f_{m,k}(\cdot) - \sum_{i=1}^N \ell_i(\cdot) f_{m,k}(\xi_i) \|^2_{L^\infty(0,\Lambda)} \leq \frac{\mu_k^{3-m} \Lambda^{2N} }{4^{2N-1}(\mu_k - \Lambda)^{2(N+1)}},
\end{equation}
for $m=0,1$;  see, e.g.,~\cite[Ch. 3.3]{davis:1975interpolation}. We are now in the position to give an estimate for the error terms $e_0$ and $e_1$:
\begin{lemma}
\label{th:loc_errors}
Let $t \in (0,\Lambda)$, $w_B \in tr\mathcal{V}(\Uh)$, and $\hat{Z}(t)$ be as in \eqref{eq:Zapprox}. Then the error terms in \eqref{eq:uerror} satisfy
\begin{equation*}
e_l \leq 12 \left[ 4 (\eta-1) \right]^{-N-1}  \left( \eta^{l+1} \Lambda^{l-1}  + \eta^{l+2} \Lambda^l \right)^{1/2} \| E w_B \|_{\mathcal{V}(\Uh)}
\end{equation*}

\noindent for $l=0,1$ and $\eta := \tilde{\Lambda}/ \Lambda$.
%
\end{lemma}
\noindent In~\cite{CPI:2018}, the parameter $\eta$ is called the {\em oversampling parameter}. Observe that for $\eta > 5/4$, $e_1$ and $e_0$ converge to zero as $N \to \infty$.
\begin{proof} 
As the estimates for $l=0,1$ follow from similar arguments, we only consider $l=1$. By triangle inequality, Parseval's identity, and \eqref{eq:Ziperror}, we have
\begin{equation}
  \label{eq:young}
  \begin{aligned}
  \frac{1}{2} e_1^2 & \leq \sum_{k=K(\tilde{\Lambda})+1}^\dimVUH
\mu^2_k \left( f_{1,k}(t) - \sum_{i=1}^N  \ell_i(t) f_{1,k}(\xi_i) \right)^2 \cdot \frac{c^2_{1,k}(w_B)}{\mu_k}
\\ 
& +\sum_{k=K(\tilde{\Lambda})+1}^\dimVUH
\mu_k \left( f_{0,k}(t) - \sum_{i=1}^N  \ell_i(t) f_{0,k}(\xi_i) \right)^2 \cdot c^2_{0,k}(w_B). 
\end{aligned}
\end{equation} 
We proceed to estimate the right hand side of \eqref{eq:Lagrange_ip}. Since $\mu_k \geq \eta \Lambda = \tilde{\Lambda} > \Lambda$, we have
\begin{equation*}
\frac{\mu_k^{3-m} \Lambda^{2N}}{4^{2N-1}(\mu_k - \Lambda)^{2(N+1)}} 
\leq \frac{\Lambda^{1-m}}{4^{2N-1} \eta^{2N+m-1}} \cdot \left( \frac{\mu_k }{\mu_k - \Lambda} \right)^{2(N+1)}
\end{equation*}
and $\mu_k(\mu_k-\Lambda)^{-1} = (1-\Lambda / \mu_k)^{-1} \leq \eta(\eta-1)^{-1}$, recalling $\eta > 1$. Hence,
\begin{equation}
\label{eq:apu2}
\mu_k^{1+m} \| f_{m,k}(\cdot) - \sum_{i=1}^N \ell_i(\cdot) f_{m,k}(\xi_i) \|^2_{L^\infty(0,\Lambda)} \leq \frac{\Lambda^{1-m} \eta^{3-m} }{4^{2N-1} } \cdot \left( \frac{1}{\eta-1} \right)^{2(N+1)}.
\end{equation}
Using Lemma~\ref{lemma:c1c2} and \eqref{eq:apu2} together with~\eqref{eq:young} gives
\begin{equation*}
e^2_l \leq \frac{2\eta^{l+1}}{4^{2N-1}(\eta-1)^{2N+2}}  \left( \Lambda^{l-1}\| \nabla (Ew_B) \|^2_{L^2(\Uh)} + \eta \Lambda^l  \| E w_B \|^2_{L^2(\Uh)} \right) \quad \mbox{for $l=1$}.
\end{equation*}
Carrying out similar argumentation leads to the same formula for $l=0$. Estimating the coefficient completes the proof.
\end{proof}
We conclude this subsection by using Lemma~\ref{th:loc_errors} to obtain an
upper bound for the local interpolation error:
\begin{theorem} 
  \label{cor:Ziperror}
 Let $Z, \hat{Z} : (0,\Lambda) \to \mathcal{B}( tr \mathcal{V}(\Uh;\mathbb{R}^N), \mathcal{V}(\Uh) )$ be as defined in \eqref{eq:fun_z} and \eqref{eq:Zapprox}, respectively. In addition, define $Z_U,\hat{Z}_U : (0,\Lambda) \to \mathcal{B}( tr \mathcal{V}(\Uh;\mathbb{R}^N), \mathcal{V}(U) )$ as  $Z_U(t) w = (Z(t)w)|_U$ and $\hat{Z}_U(t) w = (Z(t)w)|_U$, respectively. Then for $t \in (0,\Lambda)$
\begin{equation*}
\|  \hat{Z}_U(t) - Z_U(t) \|_{\mathcal{B}(tr\mathcal{V}(\Uh),\mathcal{V})} \leq C_{E} \; e(\eta,N),
\end{equation*}
where $C_{E} = C_{E}(\mathcal{V},U,\Uh) := \| E \|_{\mathcal{B}(tr\mathcal{V}(\Uh),\mathcal{V}(\Uh))}$ and 
\begin{equation*}
e(\eta,N) = 
12 \eta \left[ 4 (\eta-1) \right]^{-N-1} 
\left(
2 + \eta \Lambda
+
\frac{1}{\eta \Lambda} \right)^{1/2}
.
\end{equation*}
\end{theorem}
Recall that $Z_U, \hat{Z}_U$ depend implicitly on $\tilde{\Lambda}, N$. We expect the constant $C_{E}$ to be inversely proportional to the extension radius $r$. Note that for $\eta > 5/4$, increasing the number of interpolation points $N$ decreases the error exponentially.
%
%
\subsection{Low-rank approximation error} 
\label{sec:low_rank_error}

Recall the definitions of the operator $B \in \mathcal{B}(tr\mathcal{V}(\hat{U} ; \mathbb{R}^N), \mathcal{V}(U))$ in \eqref{eq:Bdef} and  $\mathcal{W}(U)$,
\begin{equation*}
B \vec{v}_B := \begin{bmatrix} Z_U(\xi_1) & \ldots & Z_U(\xi_N)  \end{bmatrix} \vec{v}_B \quad \mbox{and} \quad 
\mathcal{W}(U) :=
\mathrm{range}(\widehat{B}),
\end{equation*}
where $\widehat{B} \in \mathcal{B}( tr \mathcal{V}(\Uh;\mathbb{R}^N), \mathcal{V}(U))$ is a finite-rank operator. Next, we relate the error term in~\eqref{eq:WUerror}
%
%
to the operator norm of $B - \widehat{B}$. We define 
\begin{equation}
\label{eq:H1_2vecnorm}
\| \vec{w}_B \|_{tr\mathcal{V}(\Uh;\mathbb{R}^N) }
:= \left( \sum_{i=1}^N \| w_{B,i} \|^2_{tr\mathcal{V}(\Uh)} \right)^{1/2}
\end{equation}
and 
\begin{equation}
\| w \|_{\mathcal{V}_R(U)} :=  \left( \int_{U} |\nabla (R w) |^2 \; dx
+ 
\int_{U} w^2 \; dx \right)^{1/2}.
\end{equation}
We proceed with a technical lemma:
\begin{lemma}
\label{lemma:svdapu} For any $t\in (0,\Lambda)$ and $w \in \mathcal{V}(\Uh)$
\begin{equation*}
\| \vec{\ell}(t) w_B \|_{tr\mathcal{V}(\Uh;\mathbb{R}^N) }
\leq \frac{\Lambda_N}{\sqrt{2}} \| w \|_{H^1(\Uh)},
\end{equation*}
where $w_B = \gamma_{\partial \Uh} w$ and $\Lambda_N:=\max_{t\in[0,\Lambda]} \sum_{i=1}^N|\ell_i(t)|$ is the
Lebesgue constant related to the Chebyshev nodes $\{\xi_i\}_{i=1}^N \subset (0,\Lambda)$.
\end{lemma}

\noindent For the estimate of the Lebesgue constant see, e.g.,~\cite{Brutman:1978}.

\begin{proof}
%
%
Using definitions~\eqref{eq:H1_2vecnorm}~and~\eqref{eq:H1_2norm},
\begin{equation*}
\| \vec{\ell}(t) w_B
\|_{tr\mathcal{V}(\Uh;\mathbb{R}^N) }^2 = \| w_B \|^2_{tr\mathcal{V}(\Uh)}
\sum_{i=1}^N | \ell_i(\lambda) |^2 \leq \frac{1}{2} \| w \|_{H^1(\Uh)}^2 \sum_{i=1}^N |\ell_i(\lambda)|^2.
\end{equation*}
The proof is completed by observing that $\sum_{i=1}^N | \ell_i(t) |^2
  \leq  \left( \sum_{i=1}^N | \ell_i(t) | \right )^2 
  \leq \Lambda^2_N$. %
%
\end{proof}
We are now in the position to give an upper bound for the error term $e_U(\mathcal{W}(U))$ in~\eqref{eq:WUerror}.
\begin{theorem} 
\label{thm:trunc_err_est}
Let $\{ \xi_i \}_{i=1}^N$ be the Chebyshev nodes on $(0,\Lambda)$, $Z_U(t)$ be
as in \eqref{eq:ZUdef}, and $e_U$ as defined in \eqref{eq:WUerror}. Further, let $R \in
\mathcal{B}(\mathcal{V}(U),\mathcal{V})$ be a stitching operator as defined in
Section~\ref{sec:PUM}, and $B \in
\mathcal{B}(tr\mathcal{V}(\Uh;\mathbb{R}^N) ,\mathcal{V}(U))$ be as defined in \eqref{eq:Bdef}. For any $\widehat{B} \in \mathcal{B}(tr\mathcal{V}(\Uh;\mathbb{R}^N)
,\mathcal{V}(U))$,
\begin{equation*}
e_U(\mathcal{W}(U))^{1/2}
\leq \frac{1}{\sqrt{2}} \left[ C_E \; e(\eta,N) \| R \|_{\mathcal{B}(\mathcal{V}(U),\mathcal{V})} + \Lambda_N \| B - \widehat{B} \|_* \right],
\end{equation*}
where $\mathcal{W}(U) = \mathrm{range}(\widehat{B})$, and $C_E,e(\eta,N)$ are as defined in Theorem~\ref{cor:Ziperror}. Here we denote $\| \cdot \|_* := \| \cdot \|_{\mathcal{B}(tr\mathcal{V}(\Uh;\mathbb{R}^N),\mathcal{V}_R(U)) }$.
\end{theorem}
\begin{proof}
Let $w \in H^1(\Uh)$ and  $w_B = \gamma_{\partial\Uh}w$. Observe that
\begin{equation*} 
\begin{aligned}
 \inf_{v \in \mathcal{W}(U)} & \left( \int_U  |\nabla [R( Z_U(t) w_B - v)]|^2 \; dx  \right)^{1/2}   \\ & \leq  \frac{1}{\sqrt{2}} \| R \|_{\mathcal{B}( \mathcal{V}(U),\mathcal{V})} \| ( Z_U(t) - \hat{Z}_U(t) ) \|_{\mathcal{B}(tr\mathcal{V}(\Uh),\mathcal{V})} \| w \|_{H^1(\Uh)} 
 \\ & +
\inf_{v \in \mathcal{W}(U)} \left( \int_U |\nabla [R(\hat{Z}_U(t)w_B - v)] |^2 \; dx \right)^{1/2}
\end{aligned}
\end{equation*}
The first term on the right hand side is bounded by using Theorem~\ref{cor:Ziperror}. Choosing $v=\widehat{B} \vec{\ell}(t) w_B \in \mathcal{W}(U)$ and recalling $\hat{Z}_U(t) w_B = B \vec{\ell}(t) w_B$ yields
%
%
%
\begin{equation*}
 \inf_{v \in \mathcal{W}(U)} \left( \int_U |\nabla [R ( \hat{Z}_U(t)w_B - v )]|^2 \; dx \right)^{1/2} \leq \| B - \widehat{B} \|_* \| \vec{\ell}(t){w_B} \|_{tr\mathcal{V}(\Uh ; \mathbb{R}^N)}. 
\end{equation*}
Lemma \ref{lemma:svdapu} completes the proof.
\end{proof}
We have now constructed the local subspace $\tV(U)$ and estimated the local approximation error $\mathcal{E}(u,U)$ for a subdomain $U = U^\pp \subset \Omega$ via~\eqref{eq:local_approximation_error_estimate}. The error estimate for the global reduced problem follows by using the stitching operators.

\section{Partition of Unity CPI}
\label{sec:pu-cpi}
We proceed to define the local subspaces $\tV(\Up)$ used in the PU-CPI method and to derive a relative eigenvalue error estimate.  

We extend the notation of Section~\ref{sec:local-subspace} to the case of several subdomains $\{ U^\pp \}_{p=1}^M$, and we set $U=U^\pp$ for $p \in \{1,\ldots,M \}$. Denote the $r$-extension of $\Up$ by $\Uhp$ as in~\eqref{eq:kaulus}. Let $(\mu_k^\pp,v_k^\pp) \in \mathbb{R}^+ \times \mathcal{V}_0(\Uhp) \setminus \{0 \}$ satisfy
\begin{equation*}
\int_{\Uhp} \nabla v^\pp_k \cdot \nabla w \; dx = \mu_k^\pp \int_{\Uhp} v^\pp_k w \; dx, 
\end{equation*}
for each $w \in \mathcal{V}_0(\Uhp)$ as in~\eqref{eq:eigen_vk2}. We further require $\{ v_k^\pp \}$ to be an $L^2(\Uhp)$-orthonormal set, that $\mu_k^\pp$
are enumerated in non-decreasing order, and $K^\pp(\tilde{\Lambda}) := \#\{ \: k \in \mathbb{N} \; |  \; \mu^\pp_k \leq \tilde{\Lambda} \; \}$. Similarly to~\eqref{eq:VU_def}, the local subspaces are $\tV(\Up) = E_{\tilde{\Lambda}}(\Up) \oplus \mathcal{W}(\Up)$, where 
\begin{equation}
\label{eq:EppDef}
    E_{\tilde{\Lambda}}(\Up) = \mathrm{span}\{v^\pp_1|_\Up,\ldots,v^\pp_{K^\pp}|_\Up\}.
\end{equation}
Define $Z^\pp :(0,\Lambda) \rightarrow \mathcal{B}( tr\mathcal V(\Uhp), \mathcal V_0(\Uhp))$ by replacing $\mu_k$,~$v_k$, and $E$ in  \eqref{eq:ZUdef} by $\mu_k^\pp$,~$v_k^\pp$, and the right inverse of the trace operator $E^\pp: tr\mathcal{V}(\Uhp) \rightarrow \{ \; v \in \mathcal{V}(\Uhp) \; | \; v|_\Up = 0 \; \}$. Recall that the existence of $E^\pp$ is a structural assumption made on $\mathcal V$, $U^\pp$, and $\Uhp$. 

Let $Z_{\Up} : (0,\Lambda) \to \mathcal{B}(tr\mathcal{V}(\Uhp),\mathcal{V}(\Up))$  be defined as in \eqref{eq:ZUdef} and
\begin{equation}
\label{eq:defBpp}
\begin{aligned}
B^\pp &\in \mathcal{B}( tr\mathcal{V}(\Uhp; \mathbb{R}^N),\mathcal{V}(\Up) ) \quad \mbox{such that}\\
B^\pp &= \begin{bmatrix} Z_{\Up}(\xi_1) & \ldots & Z_{\Up}(\xi_N) \end{bmatrix}. 
\end{aligned}
\end{equation}
We choose the complementing subspace as $\mathcal{W}(\Up) = \mathrm{range}(\widehat{B}^\pp)$, where $\widehat{B}^\pp \in \mathcal{B}(tr\mathcal{V}(\Uhp;\mathbb{R}^N), \mathcal{V}(\Up))$ will later be a low-rank approximation of $B^\pp$. 
\begin{assumptions}\label{ass:fundamental}
Let $\Lambda>0$ and $(\lambda_j,u_j) \in (0,\Lambda) \times \mathcal{V}$ satisfy~\eqref{eq:cont_eigen}. Make the same assumptions as in Proposition~\ref{prop:eigen_proj_estimate}. Let $\widetilde{\Lambda} = \eta \Lambda$ for $\eta > 1$, and let $\{ \xi_i \}_{i=1}^N$ be the Chebyshev interpolation points of $(0,\Lambda)$.
\end{assumptions}
\begin{theorem}\label{thm:main1} Make Assumptions~\ref{ass:fundamental}. For $p=1,\ldots,M$, let $E_{\tilde{\Lambda}}(\Up)$ be as defined in \eqref{eq:EppDef}, $R^\pp \in
  \mathcal{B}( \mathcal{V}(U^\pp),\mathcal{V})$ satisfy the assumptions of  Section~\ref{sec:PUM}, $B^\pp$ be as defined in~\eqref{eq:defBpp}, and $\widehat{B}^\pp \in \mathcal{B}( tr
  \mathcal{V}(\Uhp;\mathbb{R}^N),\mathcal{V}(U^\pp))$. Define the PU-CPI method subspace $\tV$ as in \eqref{eq:PUMdef} using the local subspaces $\tV(\Up) =
  E_{\tilde{\Lambda}}(\Up) \oplus \mathcal{W}(\Up)$ and the local complementing subspaces $ \mathcal{W}(\Up) = \mathrm{range}(\widehat{B}^\pp)$.
  
  Then there exists $\tilde{\lambda} \in \sigma(\tV)$ such that
\begin{equation*}
\frac{| \lambda_j - \tilde{\lambda} |}{\lambda_j} \leq  C_M(\lambda_j) \max_{p=1,\ldots,M} \left[ \; \Lambda^2_N \| B^\pp - \widehat{B}^\pp \|^2_* + C_{E^\pp}^2 e(\eta,N)^2 \| R^\pp \|^2_{\mathcal{B}(\mathcal{V}(\Up),\mathcal{V})}   \right],
\end{equation*}
where $e(\eta,N)$ and $\| \cdot \|_*$  are as defined in Theorems~\ref{cor:Ziperror} and \ref{thm:trunc_err_est}, respectively. The constants $C_M(\lambda_j)$ and $C_{E^\pp}$ are defined as
\begin{equation*}
C_M(\lambda_j) := C(\lambda_j) (\lambda_j+1) \| \widehat{G} \|^4_{L^\infty(\Omega)} \quad \mbox{and} \quad C_{E^\pp} := \| E^\pp \|_{\mathcal{B}(tr \mathcal{V}(\Uh^\pp),\mathcal{V}(\Uh^\pp))},
\end{equation*} 
where $C(\lambda_j)$ is as defined Proposition~\ref{prop:eigen_proj_estimate}. The counting function $\hat{G}: \Omega \to \{1,\ldots,M \}$ is defined as  $\widehat{G}(x) := \#\{ p \;|\; x \in \Uhp \}$.
\end{theorem}
\begin{proof} Proposition~\ref{prop:eigen_proj_estimate} together with \eqref{eq:pum_error} and \eqref{eq:local_approximation_error_estimate} gives
\begin{equation}
\label{eq:main_th_1}
\begin{aligned}
\frac{| \lambda_j - \tilde{\lambda} |}{\lambda_j} 
&\leq 
C(\lambda_j) \| G \|^2_{L^\infty(\Omega)}
\sum_{p=1}^M \mathcal{E}(u_j,U^\pp)  \\
& \leq C(\lambda_j) \| G \|^2_{L^\infty(\Omega)} \sum_{p=1}^M  e_U \left( \mathcal{W}(U^\pp) \right) \;  \|
u_j|_\Uhp \|^2_{H^1(\Uhp)} \\
\end{aligned}
\end{equation}
for the local complementing subspaces $\mathcal{W}(U^\pp) = \mathrm{range} (\hat{B}^\pp )$ constructed in Section~\ref{sec:local-subspace} for $U = \Up$. Estimating the sum similarly with 
\eqref{eq:pum_error} and observing that $\| u_j \|^2_{H^1(\Omega)} = (\lambda_j+1) \| u_j \|^2_{L^2(\Omega)} = (\lambda_j+1)$, gives
\begin{equation}
\frac{| \lambda_j - \tilde{\lambda} |}{\lambda_j} 
\leq C(\lambda_j)  (\lambda_j+1) \| G \|^2_{L^\infty(\Omega)}  \| \widehat{G} \|^2_{L^\infty(\Omega)} \max_{p=1,\ldots,M}  e_U \left( \mathcal{W}(U^\pp) \right).
\end{equation}
Since $\Up \subset \Uhp$ we have $ \| G \|_{L^\infty(\Omega)} \leq  \| \widehat{G} \|_{L^\infty(\Omega)}$. Theorem~\ref{thm:trunc_err_est}  completes the proof.
\end{proof}
In the practical application of the PU-CPI method, the foremost challenge is to define the low-rank approximating operators $\widehat{B}^\pp$ and to efficiently construct a basis for the local complementing subspaces $\mathcal{W}(\Up)=\mathrm{range}(\widehat{B}^\pp)$. In Section~\ref{sec:fem}, we use the finite element method, i.e., $\mathcal{V} = \mathcal{V}_h$, and use singular value decomposition for this purpose. 

\section{Finite element realisation of PU-CPI}
\label{sec:fem}
Define the set function (i.e., \emph{open interior of closure}) $\mathrm{intc} : A \mapsto B$ as $A = \mathrm{int}(\overline{B})$ for $B \subset \mathbb{R}^d$.  A finite family of sets $\{ K_i \}_i \subset \Omega$ is called  a triangular or a tetrahedral partition of $\Omega$, if $K_i \subset \Omega$ are open simplicial sets satisfying $\Omega = \mathrm{intc}( \cup_i K_i )$ and $K_i \cap K_j = \emptyset$ for $i \neq j$. We make a standing assumption that partitions do not contain hanging nodes.

We consider the FE discretisation of~\eqref{eq:eigen_vk} under the following assumptions.
\begin{assumptions} 
\label{ass:RP} 
\noindent \begin{enumerate}
    \item[(i)] Let $\{ \mathcal{T}_h \}_{h}$ be a family of shape regular triangular or a tetrahedral partitions of $\Omega$ with mesh size $h = \max_{K \in\mathcal{T}_h} \mathrm{diam}(K)$ in the sense of \cite{Braess:2007}. 
    \item[(ii)] Let 
    \begin{equation}
    \label{eq:FEspace}
    \mathcal{V} = \mathcal{V}_h = \{ \; w \in H^1_0(\Omega) \; | \; w|_K \in P^1(K) \quad \mbox{for all } K \in \mathcal{T}_h \; \},
    \end{equation}
    and $\{ \psi_{l}
    \}_{l}$ the nodal basis functions of $\mathcal{V}_h$.  
\end{enumerate}
\end{assumptions}
We call $\vec{x}$ the \emph{coordinate vector} of $w\in \mathcal{V}_h$ and define the one-to-one correspondence $\vec{x} \sim w$ where $w = \sum_{l} x_l \psi_l$. The same convention is used in all subspaces of $\mathcal{V}_h$. 

An open cover $\{ U^\pp \}_{p=1}^M$ is constructed by dividing the vertices of the partition $\mathcal{T}_h$ into nonempty disjoint sets $\{ \mathcal{N}_p \}_{p=1}^M$ 
using, e.g., METIS~\cite{METIS:1998}. The set $\Up$ is  obtained as\footnote{Observe that sets $\{ \Up \}_p$ consist of simplices in partition of $\mathcal{T}_h$. Thus the diameter of each $\Up$ is always larger than $h$, linking the scale $h$ and scales of $\Up$'s.}
\begin{equation}
\label{eq:UFEdef}
U^\pp = \mathrm{intc} \{ \; K \in \mathcal{T}_h \; | \; \mbox{$K$ has at least one vertex index in $\mathcal{N}_p$} \; \} 
\end{equation} 
%
The $r$-extension of a subdomain $U^\pp$ is chosen as
\begin{equation}
\label{eq:Uextdef}
\Uh^\pp = \mathrm{intc}  \left\{ K \in \mathcal{T}_h \;\left|\; dist(K, U^\pp) \leq r \right. \right\} .
\end{equation}
An example of an open cover and the related $r$-extensions is given in Figure~\ref{fig:the_cover}. Note that our definition allows very exotic open covers, not all of which are computationally meaningful. 
%
\begin{figure}
\label{fig:the_cover}
  \centering
  \includegraphics[width=.5\textwidth]{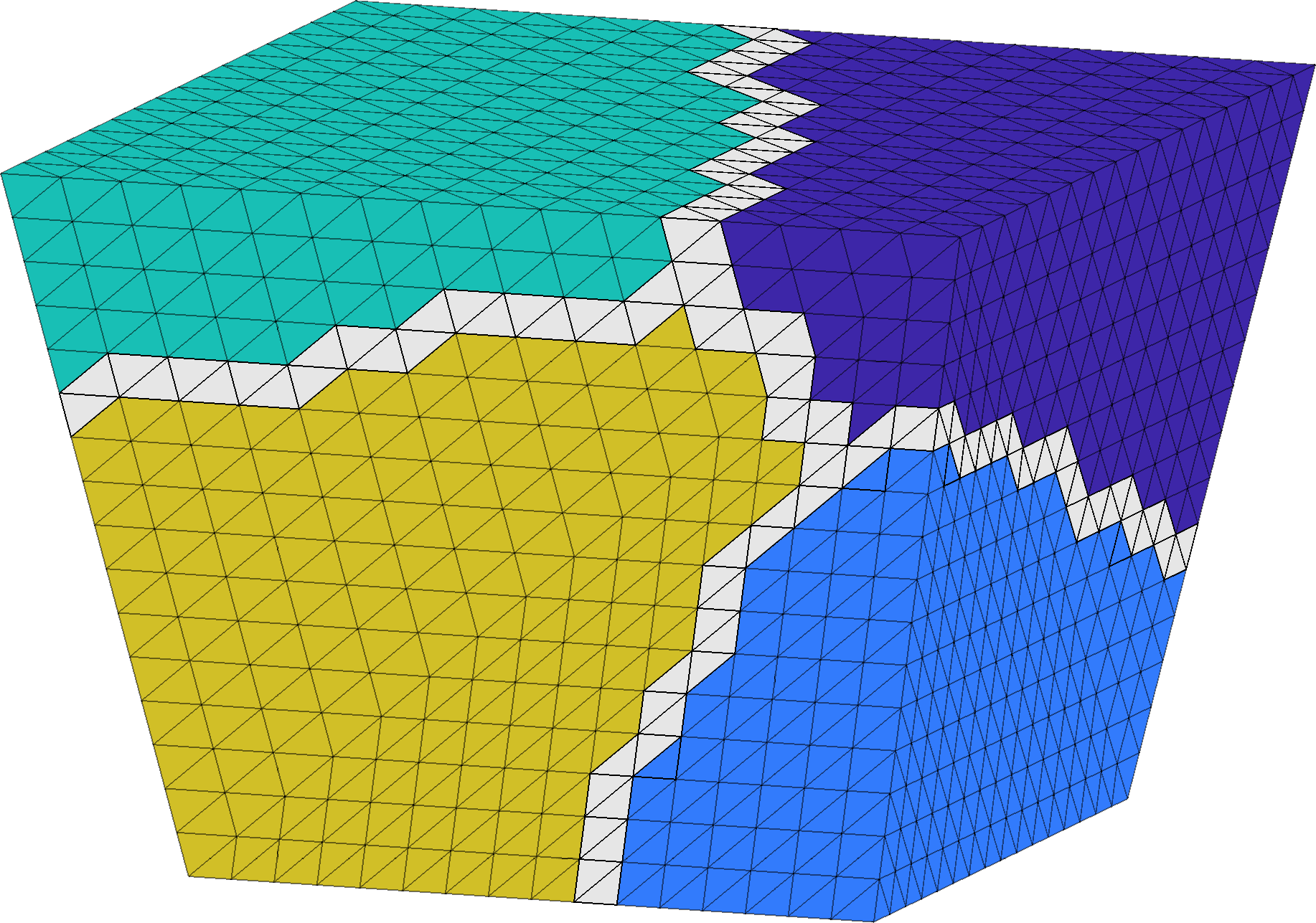} \hfill
  \includegraphics[width=.45\textwidth]{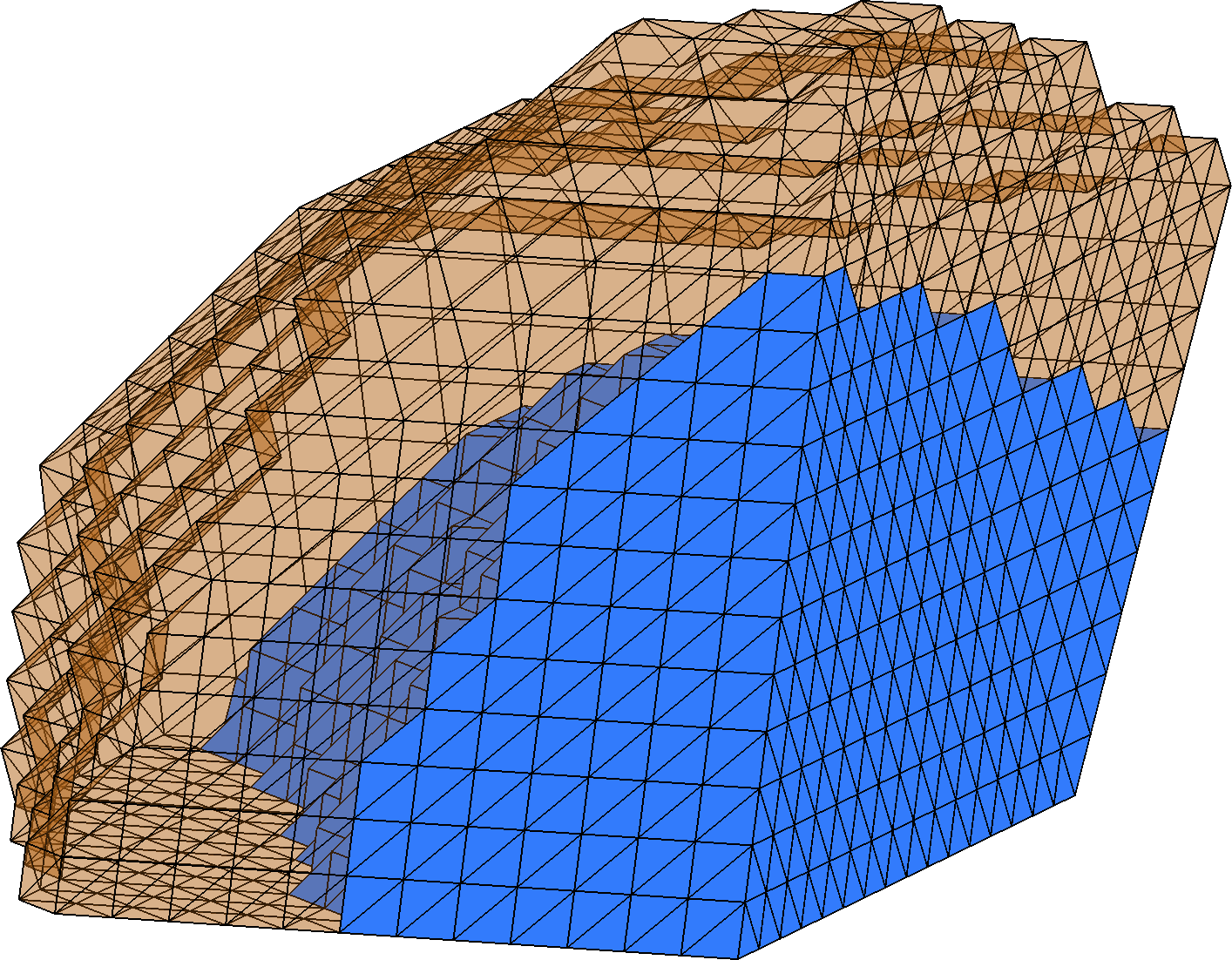}
  \caption{A partitioning of a cuboid with four subdomains
    and a visualisation of an extended subdomain on one part. Surface triangles belonging to several $U^\pp$ and to set $\Gamma$ defined in~\eqref{eq:defGamma} are visualised in white.}
\end{figure}

We proceed to define bases for the subspaces defined on $U \equiv \Up$ and $\Uh \equiv \Uhp$:
\begin{equation}
\label{eq:loc_basis}
\begin{aligned}
\mathcal{V}_h(\Uh) &= \mathrm{span} \{  \psi^{\Uh}_1 ,\ldots, \psi^{\Uh}_{\hat{n}} \}, & 
\mathcal{V}_h(U) &= \mathrm{span} \{ \psi^{U}_1, \ldots, \psi^{U}_n \}.
\end{aligned}
\end{equation}
We further assume that the basis functions are ordered so that
\begin{equation}
\label{eq:psi_ordering}
\begin{aligned}
tr \mathcal{V}_h(\Uh) &= \mathrm{span} \{ \psi^{\Uh}_1|_{\partial \Uh} ,\ldots,\psi^{\Uh}_{\hat{n}_B}|_{\partial \Uh} \}, &
tr \mathcal{V}_h(U) &= \mathrm{span} \{ \psi^{U}_1|_{\partial U} ,\ldots,\psi^{U}_{n_B}|_{\partial U} \}, \\
\mathcal{V}_{h0}(\Uh) &= \mathrm{span} \{ \psi^{\Uh}_{\hat{n}_B+1}, \ldots,  \psi^{\Uh}_{\hat{n}} \}, & \mathcal{V}_{h0}(U) &= \mathrm{span} \{ \psi^{U}_{n_B+1}, \ldots, \psi^{U}_{n} \}.
\end{aligned}
\end{equation}
Denote $n_I = n-n_B$ and $\hat{n}_I = \hat{n} - \hat n_B$ and assume that $n_I$, $\hat{n}_B$, $\hat{n}_I$, and $n_B$ all are non-zero. Because of the ordering in~\eqref{eq:psi_ordering}, it is natural to split the coordinate vectors $\vec{x} \in \mathbb{R}^{\hat{n}}$ to the boundary and interior coordinates as
\begin{equation}
\label{eq:deco_vec}
\vec{x} := \begin{bmatrix} \vec{x}_B \\ \vec{x}_I \end{bmatrix}  \quad \textrm{where} \quad \vec{x}_B \in \mathbb{R}^{\hat{n}_B} \quad \textrm{and} \quad \vec{x}_I \in \mathbb{R}^{\hat{n}_I}.
\end{equation}
 This splitting is applied to ${\hat{n}} \times {\hat{n}}$-matrices as follows
\begin{equation}
\label{eq:deco_mat}
\mathsf{A} = \begin{bmatrix} \mathsf{A}_{BB} & \mathsf{A}_{BI} \\ 
\mathsf{A}_{IB} & \mathsf{A}_{II} \end{bmatrix},
\end{equation}
where $\mathsf{A}_{BB} \in \mathbb{R}^{{\hat{n}}_B \times {\hat{n}}_B},\; \mathsf{A}_{BI} \in \mathbb{R}^{{\hat{n}}_B \times {\hat{n}}_I}$,\; $\mathsf{A}_{IB} \in \mathbb{R}^{{\hat{n}}_I \times {\hat{n}}_B}$ and $\mathsf{A}_{II} \in \mathbb{R}^{{\hat{n}}_I \times {\hat{n}}_I}$.
Let $E_h : tr \mathcal{V}_h(\Uh) \rightarrow \mathcal{V}_h(\Uh)$ be defined as
\begin{equation*}
E_h w_B = \sum_{l=1}^{{\hat{n}}_B} x_{Bl} \psi^{\Uh}_l \quad \mbox{where}
\quad \vec{x}_B = \begin{bmatrix} x_{B1} \\ \vdots \\ x_{B{\hat{n}}_B} \end{bmatrix} \sim w_B.   
\end{equation*}
That is, $E_h$ is a right inverse of the trace operator that satisfies  $(E_h w_B)|_{U} = 0$.

\subsection{Evaluation of the trace norm} 
%
%
%
%
We discuss evaluation of the norm of $tr\mathcal{V}_h(\Uh)$ required to construct $\widehat{B}$ in practice.
%
\begin{lemma} 
\label{lemma:H1_2norm} Let $\mathcal{V}_h(\Uh)$, $tr \mathcal{V}_h(\Uh)$ be as defined in \eqref{eq:loc_basis} and assume that \eqref{eq:psi_ordering} holds. Define  $\mathsf{K} \in \mathbb{R}^{\hat{n} \times \hat{n} }$ as
\begin{equation*}
\mathsf{K}_{ij} = \int_{\Uh} \left(\nabla \psi^\Uh_i \cdot \nabla \psi^\Uh_j + \psi^\Uh_i \psi^\Uh_j \right)\; dx \quad \mbox{ for } i,j = 1,\ldots,\hat{n}
\end{equation*}
and then split $K$ into $\mathsf{K}_{BB},\mathsf{K}_{BI},$ and $\mathsf{K}_{II}$ according to \eqref{eq:deco_mat}. Then for any $f \in tr \mathcal{V}_h(\Uh)$,
\begin{equation}
\label{eq:H1_2_to_scur}
\| f \|_{tr \mathcal{V}_h(\Uh)} = \left( \vec{x}_B^T \mathsf{S} \vec{x}_B \right)^{1/2} \quad \mbox{where} \quad \vec{x}_B \sim f \quad \mbox{and} \quad \mathsf{S} = \mathsf{K}_{BB} - \mathsf{K}_{BI} \mathsf{K}_{II}^{-1} \mathsf{K}_{BI}^T.
\end{equation}
\end{lemma}
\begin{proof} Observe that for $v_1,v_2 \in \mathcal{V}_h(\Uh)$ it holds
\begin{equation} 
\label{eq:unitary_eq}
(v_1,v_2)_{H^1(\Uh)} = \vec{v}_2^T K \vec{v}_1 \quad \mbox{where $\vec{v}_1 \sim v_1$, $\vec{v}_2 \sim v_2$.}
\end{equation}
Using the splitting \eqref{eq:deco_vec} and unitary equivalence \eqref{eq:unitary_eq} gives
\begin{equation*}
\| f \|^2_{tr{V_h}} = \frac{1}{2} \min_{\vec{y}_I \in \mathbb{R}^{\hat{n}_I }} \begin{bmatrix}  \vec{x}_B & \vec{y}_I \end{bmatrix} \begin{bmatrix} \mathsf{K}_{BB} & \mathsf{K}_{BI} \\ \mathsf{K}_{BI}^T & \mathsf{K}_{II} \end{bmatrix}  \begin{bmatrix} \vec{x}_B \\ \vec{y}_I \end{bmatrix} 
\quad \mbox{where $\vec{x}_B \sim f$}.
\end{equation*}
Direct calculation gives $\vec{y}_I = -\mathsf{K}_{II}^{-1} \mathsf{K}^T_{BI} \vec{x}_B$. Hence,
  \begin{equation*}
    \| f \|_{tr \mathcal{V}_h}^2 = 
    \vec{x}_B^T
    \begin{bmatrix}
      \mathsf{I} & -\mathsf{K}_{BI} \mathsf{K}_{II}^{-1}
    \end{bmatrix}
        \begin{bmatrix}
      \mathsf{K}_{BB} & \mathsf{K}_{BI} \\
      \mathsf{K}_{BI}^T & \mathsf{K}_{II}
    \end{bmatrix}
    \begin{bmatrix}
      \mathsf{I} \\ -\mathsf{K}_{II}^{-1}\mathsf{K}_{BI}^T
    \end{bmatrix}
    \vec{x}_B =
    \vec{x}_B^T \mathsf{S} \vec{x}_B,
  \end{equation*}
 which completes the proof.
\end{proof}
\begin{remark}
\label{remark:Sinv}
The matrix $\mathsf{S}$ defined in \eqref{eq:H1_2_to_scur} is dense and expensive to
construct. To circumvent this, consider the linear system
\begin{equation*}
\begin{bmatrix}
\mathsf{K}_{BB} & \mathsf{K}_{BI} \\ \mathsf{K}_{BI}^T & \mathsf{K}_{II} 
\end{bmatrix} 
\begin{bmatrix}
\vec{y}_B \\ \vec{y}_I
\end{bmatrix}
= 
\begin{bmatrix}
\vec{x}_B \\ 0 
\end{bmatrix}.
\end{equation*}
By direct calculation $\mathsf{S} \vec{y}_B =  \vec{x}_B$. Since $\mathsf{K}$ is invertible, so is $S$. Hence, 
\begin{equation}
  \label{eq:defFB}
\mathsf{S}^{-1} \vec{x}_B = \mathsf{F}_B^T\mathsf{K}^{-1}\mathsf{F}_B \vec{x}_B \quad \mbox{where} \quad \mathsf{F}_B \in \mathbb{R}^{n_B \times n},\; (\mathsf{F}_B)_{ij} = \delta_{ij}.
\end{equation}  
Using the equation above, the action of $\mathsf S^{-1}$ can be efficiently computed by storing the Cholesky factorisation of $\mathsf K$. Due to this, our implementation of PU-CPI method subspace uses $\mathsf S^{-1}$ instead of $\mathsf S$.
\end{remark}

\subsection{Stitching operators}  in Section~\ref{sec:PUM}, the open cover $\{ \Up \}_p$ is related to a family of stitching operators $\{ R^\pp \}_p$ , $R^\pp : \mathcal{V}(U^\pp) \to \mathcal{V}$. For $\mathcal{V} = \mathcal{V}_h$ we define the stitching operator $R_h : \mathcal{V}_h(U) \rightarrow \mathcal{V}_{h}$ corresponding the subdomain $U=U^\pp$ by
\begin{equation}
\label{eq:RFEdef}
(R_h)|_{\Omega \setminus U} = 0 \quad \textrm{and} \quad (R_h w)|_{U} = \sum_{l=n_B+1}^{n} \psi^U_{l} x_l.
\end{equation}
Even though 
\footnote{In our implementation of the stitching operator, we select the basis functions $\{ \psi^U_{l} \}_l$ from the set $\{  \psi_{l} \}_l$ to avoid changing bases. Keeping track of the related indexing is challenging and not discussed here nor in the following.}
$\{  \psi^U_{l} \}_{l=n_B+1}^n$ is a basis of $\mathcal{V}_{h0}(U)$, the embedding $\mathcal{V}_{h0}(U)$ into $\mathcal{V}_h$ by zero extension makes it possible to regard $R_h w$ as element of $\mathcal{V}_h$. The PU-CPI error estimate in Theorem~\ref{thm:main1} depends on $\| R_h \|_{\mathcal{B}(\mathcal{V}_h(U),\mathcal{V}_h)}$, which we estimate next.
\begin{lemma} 
\label{lemma:Rhbound}
Let $U \subset \Omega$ be defined similarly to \eqref{eq:UFEdef} and $R_h \in \mathcal{B}(\mathcal{V}_h(U),\mathcal{V}_h)$  as in \eqref{eq:RFEdef}. Under Assumptions~\ref{ass:RP} there exists constant $C_R=C_R(\{ \mathcal{T}_h \}_h )$ such that 
\begin{equation*}
\| R_h \|_{\mathcal{B}(\mathcal{V}_h(U),\mathcal{V}_h)} \leq C_R h^{-1}.
\end{equation*}
\end{lemma}
\begin{proof}
    Recall that $\mathcal{V}_h$ and $\mathcal{V}_h(U)$ inherit their norms from $H^1_0(\Omega)$ and $H^1(U)$, respectively. Let $w\in \mathcal{V}_h(U)$ and $\vec{x} \sim w$. By the inverse inequality in,~e.g.,~\cite[Section 4.5]{brenner_scott:1994} there exists  constant $C_{inv}:= C_{inv}(\{\mathcal{T}_h \}_h)$, independent of $h$, such that 
\begin{equation*}
\| R_h w \|_{H_0^1(\Omega)} = \| \nabla R_h w \|_{L^2(\Omega;\mathbb{R}^d)} 
\leq C_{inv} h^{-1} \| R_h w \|_{L^2(\Omega)}.
\end{equation*}
Observe that $\mathrm{supp}(R_h w) \subset U$ for each $w \in \mathcal{V}_h(U)$. The following norm equivalence is given, e.g., in \cite[Lemma 6.2.7]{brenner_scott:1994}: 
\begin{equation}
\label{eq:norm_equiv}
c_1 h^{d/2} | \vec{x} |
\leq 
\| w \|_{L^2(U)}
\leq 
C_1 h^{d/2} | \vec{x} | \quad \mbox{where} \quad |\vec{x}| = (\vec{x}^T \vec{x})^{1/2} 
\end{equation}
for any $w\in \mathcal{V}_h(U)$, $\vec{x} \sim w$, and constants $c_1=c_1(\{\mathcal{T}_h\}_h), C_1=C_1(\{\mathcal{T}_h\}_h)$. Using \eqref{eq:norm_equiv} and the definition \eqref{eq:RFEdef} gives
\begin{equation*}
\| R_h w \|_{L^2(\Omega)} 
\leq C_1 h^{d/2} \left( \sum_{l = n_B+1}^{n} x_l^2 \right)^{1/2} \leq
C_1 h^{d/2} |\vec{x}| \leq C_1 c^{-1}_1 \| w \|_{L^2(\Omega)}.
\end{equation*}
\phantom{a}
\end{proof}

\subsection{The local complementing subspace}

We proceed to construct a basis for the local complementing subspace $\mathcal{W}_h(U)$. To this
end, we represent the linear operators $Z_h(t)$ and $B_h$ as matrices using the
bases of $tr\mathcal{V}_h(\Uh),\mathcal{V}_{h0}(\Uh)$ and $\mathcal{V}_h(U)$ defined in \eqref{eq:loc_basis}--\eqref{eq:psi_ordering}. Denote by $\mathsf{A},\mathsf{M} \in \mathbb{R}^{\hat n \times \hat n}$ the stiffness and mass matrices of the FE-discretised version of~\eqref{eq:CPIstart}, respectively. Both of these matrices are splitted as in~\eqref{eq:deco_mat}.
Following Section~\ref{sec:evaluation-z}, the matrix representation of $Z_h(t)$
is $\mathsf{Z}_h : (0,\Lambda) \rightarrow \mathbb{R}^{\hat{n}_I \times \hat{n}_B}$ given by 
\begin{equation}
\label{eq:Zhdef}
\mathsf{Z}_h(t)  := \mathsf{P}_h (\mathsf{A}_{II} - t \mathsf{M}_{II})^{\dagger}(-\mathsf{A}_{BI}^T + t \mathsf{M}_{BI}^T),
\end{equation}
where $\mathsf{Z}_h(t)$ is real analytic for all $t \in (0,\Lambda)$. Here $\dagger$ is the Moore-Penrose pseudo-inverse and
$\mathsf{P}_h := \mathsf I - \sum_{k=1}^{K(\tilde{\Lambda})} \vec{v}_k\vec{v}_k^T \mathsf{M}_{II}$, 
where $\vec{v}_k \sim v_k$ for eigenfunctions $v_k \in \mathcal{V}_h(\Uh)$ of \eqref{eq:eigen_vk2} 
\footnote{This is another way to define $Z_h$ for all $t \in (0,\Lambda)$ compared to Section~\ref{sec:eigenf-repr-form}, also used in \cite{CPI:2018}.}. The matrix representation of the
operator $B_h$, defined in \eqref{eq:Bdef}, in the natural basis of the cartesian product space $tr \mathcal{V}_h(\Uh;\mathbb{R}^N)$ is
%
\begin{equation}
\label{eq:FUdef}
\mathsf{B}_h =
\mathsf{F}_U  \begin{bmatrix} \mathsf{Z}_h(\xi_1) & \cdots & \mathsf{Z}_h(\xi_N) \end{bmatrix} \in \mathbb{R}^{n \times N \hat{n}_B},
\end{equation}
where $n = \mathrm{dim}(\tV_h(U))$ and $\mathsf{F}_U \in \mathbb{R}^{{\hat{n}_I} \times n}$ is the matrix representation of the restriction operator $F_U: \mathcal{V}_h(\Uh) \to \mathcal{V}_h(U)$ given by $F_U v = v|_U$ in bases \eqref{eq:loc_basis}--\eqref{eq:psi_ordering}. The norm of the Cartesian product space $tr\mathcal{V}_h(\Uh;\mathbb{R}^N)$ in terms of coordinate vectors is given by
\begin{equation*}
\| \vec{v_B} \|_{tr\mathcal{V}_h(\Uh;\mathbb{R}^N)} =  \| (\mathsf{I}_N \otimes \mathsf{S}^{1/2}) \vec{x}_B \|_2 \quad \mbox{for} \quad \vec{x}_{B} \sim v_{B}
\end{equation*}
by Lemma~\ref{lemma:H1_2norm}. Here  $\mathsf{I}_N \in \mathbb{R}^{N \times N}$ is the identity matrix and $\otimes$ denotes the Kronecker product. Finally, observe that $\| w \|_{\mathcal{V}_{hR}(U)} =  \| \mathsf{K}_R^{1/2} \vec{x} \|_2 $  with $\vec{x} \sim w$ and the symmetric, positive definite matrix $\mathsf{K}_R \in \mathbb{R}^{n \times n}$ defined as
\begin{equation}	
\label{eq:KRdef}
(\mathsf K_R)_{lm} = \int_U \left( \nabla (R_h \psi^U_l) \cdot \nabla (R_h \psi^U_m) + \psi^U_l \psi^U_m \right) \; dx.
\end{equation}
It is well--known that the finite--dimensional operator $B_h  \in \mathcal{B}(tr\mathcal{V}_h(\Uh;\mathbb{R}^N),\mathcal{V}_{hR}(U))  $ has the singular values $\sigma_1 \geq \sigma_2 \geq \ldots \geq \sigma_{n} \geq 0$ and for $k < n$ there exists rank $k$ operators $B_{hk}$ satisfying
\begin{equation}
\label{eq:cor_min1}
\min_{ \substack{\mathrm{rank}(T) \leq k}}  \| B_h - T \|_* = \| B_h - B_{hk} \|_* = \sigma_{k+1},
\end{equation}
where $\| \cdot \|_* = \| \cdot \|_{\mathcal{B}(tr\mathcal{V}_h(\Uh;\mathbb{R}^N),\mathcal{V}_{hR}(U)) }$. Here, we have used the fact that $n < \hat{n}$. These operators are obtained by computing the SVD of the $\mathbb{R}^{n \times N \hat{n}}$--matrix 
\begin{equation*}
\mathsf{C}:= \mathsf K_R^{1/2} \mathsf B_h ( \mathsf I_N \otimes \mathsf S^{-1/2}) = \sum_{l=1}^n \sigma_l \vec{u}_l \vec{v}_l^T,
\end{equation*}
where $\{ \vec{u}_l \}_{l=1}^{n} \subset \mathbb{R}^n$ and $\{ \vec{v}_l \}_{l=1}^{n} \subset \mathbb{R}^{N \hat{n}}$ are left-- and right--singular vectors of $\mathsf{C}$, respectively. Then
\begin{equation}
\label{eq:low_rank_Bhk}
\mathsf{B}_{hk} = \mathsf{K}_R^{-1/2} \left( \sum_{l=1}^k \sigma_l \vec{u}_l \vec{v}_l^T \right) (\mathsf{I}_N \otimes \mathsf S^{1/2}),
\end{equation}
as can be seen from the definition of the operator norm $\| \cdot \|_*$ by a change of variables.

Let the local complementing subspace be $\mathcal{W}_h(U) = \mathrm{range}(B_{hk})$ for $B_{hk}$ given in \eqref{eq:low_rank_Bhk}. The basis for $\mathcal{W}_h(U)$ is obtained from the first $k$ left--singular vectors $\{ \vec{u}_l \}_l$ of the matrix $\mathsf{C}$ as
\begin{equation}\label{eq:WUfem}
  \mathcal{W}_h(U) = \left\{ \; \sum_{l=1}^n y_l \psi^U_l \in \mathcal{V}_h(U) \; \left|\; \vec{y} \in \mathsf{K}_R^{-1/2} \mathrm{span}\{ \vec{u}_1, \ldots, \vec{u}_k \}  \right. \right\}.
\end{equation}
In practice, the vectors $\{ \vec{u}_l \}_l$ are computed by solving the largest $k$ eigenpairs of the $ \mathbb{R}^{n \times n}$--matrix\footnote{In practice, the square roots $K_R^{1/2}$ are replaced by the Cholesky factors of $K_R$}
\begin{equation}
  \label{eq:ckdef}
\mathsf C \mathsf C^T = \mathsf K_R^{1/2} \left( \sum_{i=1}^N \mathsf F_U \mathsf Z_h(\xi_i) \mathsf S^{-1} \mathsf Z_h(\xi_i)^T \mathsf F^T_U  \right) \mathsf K_R^{1/2}
\end{equation}
using the Lanczos iteration with the mapping $\vec{x} \mapsto \mathsf C \mathsf C^T \vec{x}$. There are two reasons for using the dual approach. First, the dimension of $\mathsf C \mathsf C^T$ is independent of $N$. Second, an explicit construction of $\mathsf{S}$ is avoided by utilising Remark~\ref{remark:Sinv}. Combining the above discussion with Theorem~\ref{thm:main1} yields an
estimate for the relative eigenvalue error.
\begin{theorem} 
\label{thm:main2} 
Make Assumptions~\ref{ass:fundamental}~and~let $\tV_h$ satisfy Assumptions \ref{ass:RP}. Let the stitching operators $R^\pp_h$, $p \in \{1,\ldots, M \}$, be defined as in \eqref{eq:RFEdef} for $U=U^\pp$.~Let the singular values $\sigma^\pp_1 \geq \sigma^\pp_2 \geq \ldots \geq \sigma^\pp_{n^\pp}$ and left--singular vectors $\{ \vec{u}_l^\pp \}_{l=1}^{n^\pp}$ be defined as above for $U=U^\pp$. The local complementing subspaces are defined as
\begin{equation*}
\mathcal{W}_h(U^\pp) :=\left\{ \sum_l y_l \psi^{U^\pp}_l \; \left|\; \vec{y} \in {\mathsf{K}_R^\pp}^{-1/2} \mathrm{span}\{ \vec{u}^\pp_1, \ldots, \vec{u}^\pp_{k^\pp} \}  \right. \right\},
\end{equation*}
where $\{k^\pp \}_{p=1}^M $ are local cut-off indices, $\{ \psi^{U^\pp}_l \}_l$ is a basis of $\mathcal{V}_h(U^\pp)$, and $K_R^\pp$ defined as in \eqref{eq:KRdef} for $U=U^\pp$ and $R=R^\pp$. Define the local subspaces as $\tV_h(\Up) := E_{\tilde{\Lambda}}(\Up) \oplus \mathcal{W}_h(U^\pp)$, where $E_{\tilde{\Lambda}}(\Up)$ is as in \eqref{eq:EppDef}, and the associated PU-CPI method subspace $\widetilde{\mathcal{V}}_h$ as in \eqref{eq:PUMdef}. 

Then there exists $\tilde{\lambda} \in \sigma(\tV_h)$ such that
\begin{equation*}
\frac{| \lambda_j - \tilde{\lambda} |}{\lambda_j} \leq C_M(\lambda_j) \max_{p=1,\ldots,M} \left[ \; \Lambda_N (\sigma^\pp_{k^\pp+1})^2  + C^2_R C^2_{E^\pp} h^{-2} e^2 (\eta,N) \; \right],
\end{equation*}
where $\Lambda_N$, $C_M \left(\lambda_j \right)$ and  $C^2_{E^\pp}$ are as defined in Theorem~\ref{thm:main1},~$C_R$ as in Lemma~\ref{lemma:Rhbound}, and $e(\eta,N)$ as in Theorem~\ref{cor:Ziperror}.
\end{theorem}

\subsection{Assembly of the PU-CPI Ritz eigenproblem} 
\label{sec:ass}
The remaining task is to solve the global Ritz eigenvalue problem~\eqref{eq:red_eigen} posed in the PU-PCI method subspace $\tV_h$. Let  $\{ \varphi_l^\pp \}_l$ be a basis of the space $R_h^\pp \tV_h(\Up) \subset \mathcal{V}_h$ and denote $n^\pp := \mathrm{dim}(R_h^\pp \tV_h(\Up))$. Then the ordered set
\begin{equation}
\label{eq:ordering}
\{ \varphi_l^\pp \; | \; l=1,\ldots, n^\pp, \; p = 1 ,\ldots, M \; \} = \{ \phi_k \; | \; k=1,\ldots, \sum_{p=1}^M n^\pp \; \}
\end{equation}
is a basis for the PU-CPI method subspace $\tV_h$ defined in \eqref{eq:PUMdef} with dimension $\tilde{n} := \sum_{p=1}^M n^\pp$. The ordering in \eqref{eq:ordering} defines an integer--valued function $\sigma(p,l)$ satisfying
\begin{equation*}
\varphi_l^\pp = \phi_{\sigma(p,l)} \quad \mbox{for $l=1,\ldots, n^\pp$, $p = 1 ,\ldots, M$}.
\end{equation*}
Next, we assemble the matrices $\mathsf{A},\mathsf{M}$ in the
global eigenproblem: find $(\tilde{\lambda}_k,\tilde{\vec{v}}_k) \in \mathbb{R}^+
\times \mathbb{R}^{\tilde{n} }$ such that
\begin{equation*}
\mathsf{A} \tilde{\vec{v}}_k = \tilde{\lambda}_k \mathsf{M} \tilde{\vec{v}}_k,
\end{equation*}
where $\mathsf{A}_{lm} = ( \nabla \phi_{l},\nabla
\phi_{m})_{L^2(\Omega;\mathbb{R}^d)}$and  $\mathsf{M}_{lm} = (
\phi_{l},\phi_{m})_{L^2(\Omega)}$.
In our early numerical experiments, a straightforward assembly of $\mathsf{A}$ and $\mathsf{M}$ proved to be time consuming. Next, we outline a more efficient and numerically more stable strategy.

We only study the entries of $\mathsf{A}$ since the entries of $\mathsf{M}$ are computed similarly.  The entries of $\mathsf{A}$ are obtained by computing
\begin{equation}
\label{eq:redAentries}
\mathsf{A}_{\sigma(p,l),\sigma(q,m)} = \int_{\Omega} \nabla \varphi^\pp_l \cdot \nabla \varphi^{(q)}_m \; dx
\end{equation} 
for each $l=1,\ldots,n^\pp$, $m=1,\ldots,n^{(q)}$ and $p,q \in \{1,\ldots,M\}$. If $p\neq q$ in \eqref{eq:redAentries},
\begin{equation}
\label{eq:redAskeleton}
\mathsf{A}_{\sigma(p,l),\sigma(q,m)} = \int_{\Gamma} \nabla \varphi^\pp_l \cdot \nabla \varphi^{(q)}_m \; dx,
\end{equation}
where the overlap set $\Gamma \subset \Omega$ is defined as
\begin{equation*}
\label{eq:defGamma}
\Gamma = \mathrm{intc} \{ \; K \in \mathcal{T}_h \; | \; K \mbox{ has vertex indices in at least two sets $\mathcal{N}_p$} \; \}, 
\end{equation*}
see Figure~\ref{fig:the_cover}. The off-diagonal entries in \eqref{eq:redAskeleton} can be
computed if the functions $\{ \varphi_l|_\Gamma \}_{l=1}^{\tilde{n}}$ are
known. 

If $p = q$ in \eqref{eq:redAentries},
\begin{equation}
\label{eq:redAdomain}
\mathsf{A}_{\sigma(p,l), \sigma(p,m) } = \int_{\Up} \nabla \varphi^\pp_l \cdot \nabla \varphi^\pp_m \; dx.
\end{equation}
To store the minimal amount of data, the basis functions $\{ \varphi_l^\pp \}_{l=1}^{n^\pp}$ are solutions of the symmetric eigenvalue problem
\begin{equation}
\label{eq:loc_red_eigen}
\int_\Up \nabla \varphi_l^\pp \cdot \nabla \varphi_m^\pp \; dx = 
d_{l}^\pp
\int_\Up \varphi_l^\pp \varphi_m^\pp \; dx \quad \mbox{and} \quad \|\varphi_l^\pp \|_{L^2(\Up)} = 1
\end{equation}
for eigenvalues $d^\pp_l \in \mathbb{R}^+$ and for each $l,m= 1,\ldots,n^\pp$.  Thus, for each $p$,
\begin{equation*}
\mathsf{A}_{\sigma(p,l),\sigma(p,m)} =  d^\pp_{l} \delta_{lm} 
\quad \mbox{and} \quad
\mathsf{M}_{\sigma(p,l),\sigma(p,m)} =  \delta_{lm}.
\end{equation*}
To summarise, the matrices $\mathsf{A}$ and $\mathsf{M}$ can be fully characterised based on the data
\begin{equation*}
\{ \varphi_{l}|_{\Gamma} \}_{l=1}^{\tilde{n}},
\quad 
\{ \nabla \varphi_l|_\Gamma \}_{l=1}^{\tilde{n}}, \quad \mbox{and} \quad \{ d^\pp_{l} \}_{l=1}^{n^\pp} \quad \mbox{for $p=1,\ldots,M$}.
\end{equation*}
If needed, restrictions of the basis functions are can be stored, e.g., on some inner surface to visualise the eigenfunctions.

\subsection{Overview of the PU-CPI algorithm}
\label{sec:algorithm}

The PU-CPI is intended for distributed computing environment with a 
single master and multiple workers. The input data for the algorithm is specified in Table~\ref{table:inputs}.

\begin{table}[ht] \centering
\caption{Input parameters to the PU-CPI algorithm}
\label{table:inputs}
\begin{tabular}{|c|l|} \hline
$\Lambda$ 	    & Spectral interval of interest $(0,\Lambda)$ \\
\hline
$N$ 			    & Number of interpolation points \\
\hline
$\eta$		    & Oversampling parameter \\
\hline
$\mathcal{T}_h$ & Triangular (d=2) or tetrahedral (d=3) partition of $\Omega$ \\
\hline
$M$			    & 	Number of subdomains \\
\hline
$r$				& Extension radius \\
\hline
$tol$			&  Cut-off tolerance for singular values \\ 
\hline
\end{tabular}
\end{table}
The cut-off tolerance is used to determine the parameters $k^\pp$ in Theorem~\ref{thm:main2} so that $\sigma^\pp_{k^\pp +1} \leq tol$. Theorem~\ref{thm:main2} gives the error estimate: for any $\lambda_j \in \sigma(\mathcal{V}_h) \cap (0,\Lambda)$ there exists $\tilde{\lambda} \in \sigma(\tV_h)$ such that
\begin{equation*}
\frac{| \lambda_j - \tilde{\lambda} |}{\lambda_j} \leq C \left[ tol^2  + e^2 (\eta,N) \right] \quad \mbox{for some constant $C$}.
\end{equation*}
The PU-CPI  proceeds in three steps:

\newcounter{AlgStepCounter}

\smallskip

\stepcounter{AlgStepCounter}
\noindent \textit{Step \arabic{AlgStepCounter}.(work division)} METIS is used to partition the vertices of $\mathcal{T}_h$ into $M$ subsets by the master. The
submeshes defining $\Up$ and $\Uhp$ are created from these vertex sets as
explained in Section \ref{sec:fem}. The submeshes defining $\Up$ and $\Uhp$ for $p=\{1,\ldots,M\}$ are submitted to workers.

\smallskip
\stepcounter{AlgStepCounter}
\noindent \textit{Step \arabic{AlgStepCounter}.(distributed computation)} Each worker receives a submesh and computes a basis for  $R_h^\pp \tV_h(\Up)$ in the following steps (i)--(v), where all matrices refer to the subdomain $U^\pp$. 

\begin{enumerate}[(i)]
\item Assemble the stiffness and mass matrices $\mathsf{A},\mathsf{M}$ related\footnote{The homogeneous Dirichlet boundary condition is imposed on $\partial \Uhp \cap \partial \Omega$ and this has been communicated to the worker. } to $\mathcal{V}_h(\Uh^\pp)$. Split $\mathsf{A},\mathsf{M}$ to interior and boundary parts according to \eqref{eq:deco_mat}. Compute the $K(\tilde{\Lambda})$ lowest eigenpairs $(\mu_k, \vec{v}_k)$ of
  the pencil $(\mathsf{A}_{II}, \mathsf{M}_{II})$, and form the projection $\mathsf{P}_h = \mathsf{I} -
  \sum_{k=1}^{K(\tilde{\Lambda})} v_k v_k^T \mathsf{M}_{II}$.
\item Construct the matrices $\mathsf K_R$ as in~\eqref{eq:KRdef}, $\mathsf F_U$ as in~\eqref{eq:FUdef}, $\mathsf K$ as in Lemma~\ref{lemma:H1_2norm}, and $\mathsf F_B$ as in~\eqref{eq:defFB}.
\item Compute the largest eigenpairs $(\sigma^2_k, \vec{c}_k)$ of 
\begin{equation*}
\mathsf C \mathsf C^T = \mathsf L^T \left( \sum_{i=1}^N \mathsf F_U \mathsf Z_h(\xi_i) \mathsf S^{-1} \mathsf Z_h(\xi_i)^T \mathsf F^T_U  \right) \mathsf L
\end{equation*}
using Lanczos iteration. The action $\vec{x} \mapsto \mathsf S^{-1} \vec{x}$ is evaluated as explained in  Remark~\ref{remark:Sinv}. 
\item An auxiliary basis for $R_h^\pp \mathcal{V}_h(\Uh^\pp)$ is obtained from column vectors of $\mathsf{Q}$,  
\begin{equation*}
\mathsf{Q} := \mathsf{R} \begin{bmatrix} 
\mathsf F_U \vec{v}_1,\ldots, \mathsf F_U \vec{v}_{K(\widetilde{\Lambda})}, \mathsf{L}^{-T}\vec{c}_1,\ldots, \mathsf{L}^{-T}\vec{c}_k
\end{bmatrix},
\end{equation*}
where $\mathsf{R}$ is the matrix representation of $R_h^\pp$ restricted to $\mathcal{V}_{h0}(U^\pp)$. To satisfy \eqref{eq:loc_red_eigen}, we solve the diagonal matrix $\mathsf D$ and the invertible matrix $\mathsf V$ from the eigenvalue problem
\begin{equation*}
\mathsf Q^T \mathsf{A}_0 \mathsf{QV}  =  \mathsf Q^T \mathsf{M}_0 \mathsf{Q V D} \quad \mbox{and} \quad \mathsf{V}^T \mathsf{Q}^T \mathsf{M}_0 \mathsf{Q} \mathsf{V} = \mathsf{I},
\end{equation*}
where $\mathsf A_0$ and $\mathsf M_0$ are the stiffness and mass matrices in $\mathcal{V}_{h0}(U^\pp)$. 
The final subspace is obtained from the columns of $\tilde{\mathsf Q} = \mathsf Q \mathsf V$.
  
\item Submit $\mathrm{diag }(\mathsf D)$ and $\tilde{\mathsf
    Q}(:,n_\Gamma)$ to the master. Here $n_\Gamma$ is set of those vertex indices that lie on
  $\overline{\Gamma}$.
\end{enumerate}

\stepcounter{AlgStepCounter}
\smallskip 
\noindent \textit{Step \arabic{AlgStepCounter}.(Solution of the global PU-CPI eigenproblem)} The master solves~\eqref{eq:red_eigen} posed in the method subspace $\tV_h$. The required matrices are constructed as outlined in Section~\ref{sec:ass} and the resulting problem solved using the Lanczos iteration.

\smallskip
\section{Numerical examples}
\label{sec:num_ex}
We give numerical examples validating the theoretical results and demonstrating the 
potential of PU-CPI variant of Section~\ref{sec:fem}. For this purpose, we use a cluster of $26$ desktop 
computers of which $24$ had a Xeon E3-1230 CPU, and two were equipped with Xeon W-2133.
There was 32 GB of RAM in all but one workstation which had 64 GB. 
Because solving the smallest eigenvalues of the global Ritz eigenvalue problem~\eqref{eq:red_eigen} posed in the 
PU-PCI method subspace $\tV_h$ using shift-and-invert Lanczos iteration requires lots of memory, the workstation with 64 GB of
RAM acted as the master. All data were transferred over NFS, and distributed tasks were launched using GNU 
parallel~\cite{GNUPARA}. All computations were done using MATLAB R2019a. As the computers were also in other use, the given run-time estimates are conservative.   



We study the behaviour and convergence of the method using the domain
\begin{equation}
\label{eq:num_domain}
\Omega = F((0,1)^3) \quad \mbox{where} \quad F : \begin{bmatrix} x_1 \\ x_2 \\ x_3 \end{bmatrix} \mapsto \begin{bmatrix}
x_1 + 0.4 x_3 ( 2 x_1 - 1) \\
x_2 + 0.4 x_3 ( 2 x_2 - 1) \\
x_3
\end{bmatrix},
\end{equation}
see Figure~\ref{fig:the_cover}.  As in Section~\ref{sec:fem}, problem~\eqref{eq:cont_eigen} is posed in the space $\mathcal{V} \equiv \mathcal{V}_h$, where $\mathcal{V}_h$ is the finite element space of piecewise linear function over tetrahedral partition $\mathcal{T}_h$ of domain $\Omega$. The mesh parameter values $h$ are varied by mapping different uniform tetrahedral meshes of $(0,1)^3$ with $F$. The open cover $\{ \Up\}_p$ of $\Omega$ is constructed by splitting the vertex indices of the $\mathcal{T}_h$ into disjoint sets $\{ \mathcal{N}_p \}_p$ using METIS as explained in Section
~\ref{sec:fem}. The subdomains produced in this manner can have significantly different shapes and sizes in a way that cannot be controlled. We observed that choosing the extension radius $r^\pp$ proportional to the diameter of the corresponding subdomain $\Up$ is beneficial for keeping the dimension of the sub--problems reasonable. This is done heuristically: define the empirical radius of $\Up$ by
\begin{equation*}
r_c^\pp = \frac{1}{2} \left( \max_{m\in \mathcal{N}_p} \vec{u}^T \vec{x}_m - \min_{m\in \mathcal{N}_p} \vec{u}^T \vec{x}_m\right),
\end{equation*}
where $\vec{u}$ is the first principal component of the coordinate vector set $\{ \vec{x}_m \}_{m\in\mathcal{N}_p}$. Unless otherwise stated, we choose the extension radius for subdomain $\Up$ as $r^\pp = 0.2 r_c^\pp$.

Intuitively speaking, we have observed that PU-CPI works best if the subdomains $\Up$ touch each other as little as possible. So as to domain $\Omega$ in \eqref{eq:num_domain}, we observed that METIS produces subdomains that have significant intersections compared to their diameters. This represents the worst--case behaviour of PU-CPI. 

Throughout this section, we approximate $200$ lowest eigenvalues of problem~\eqref{eq:cont_eigen}, and the parameter $\Lambda$ is chosen accordingly. While experimenting with PU-CPI, it appears that choosing $N=5$ and $\eta=2.5$ makes the interpolation error smaller than $10^{-10}$ for all mesh sizes $h$ used. Hence, these values were kept fixed, and the dependency of the relative eigenvalue error on $N$ and $\eta$ was not investigated. We focus on the effect of cut-off tolerance of singular values, number of subdomains, problem size, and the extension radius on computational load and accuracy.


\subsection{Varying mesh density}
\label{sec:varying-mesh-density}

The eigenvalue problem \eqref{eq:cont_eigen} was solved with different mesh parameters $h$. Subdomains with about $\num{5000}$ vertices were used except for the three densest meshes. For these meshes, a smaller number of larger subdomains was required to decrease $\mathrm{dim}(\tV_h)$, so that the eigenvalue problem~\eqref{eq:red_eigen} posed in space
 $\tV_h$ could be solved by the master workstation. Since METIS failed to partition the densest mesh, it was manually divided into cube-shaped subdomains. 

The results are shown in Table~\ref{tab:ndof_error}. The maximum relative eigenvalue error was estimated by comparing PU-CPI against shift-and-invert Lanczos solution of \eqref{eq:cont_eigen} using MATLAB's \texttt{eigs} function with a tolerance of $10^{-10}$. The sparsity of the matrices produced by PU-CPI is shown in Figure~\ref{fig:nnz}. A breakdown of time required by each step of PU-CPI is shown in Table~\ref{tab:ndof_t}.  The comparable values $\mathrm{t_{CPI}}$ and $\mathrm{t_{FEM}}$ are the wall clock times (in seconds) spent after the mesh structure was constructed. For fair comparison, standard FE solution uses MATLAB's \texttt{eigs} function with a tolerance of $10^{-4}$. In addition, $\mathrm{t_{CPI}}$ includes file I/O times and network delays, where as $\mathrm{t_{FEM}}$ includes the time required to assemble the full stiffness and mass matrices.

\subsection{Effect of subdomain extension}
\label{sec:effect-subd-extens}
When using a larger extension radius $r^\pp$, the singular values  $\sigma^\pp_k$ of $\mathsf{C}$ in~\eqref{eq:ckdef} are expected to decay faster. This effect is studied using meshes with $\num{54872}$ and $\num{195112}$ of Degrees--Of--Freedom (DOF). In both cases, the singular values $\sigma^\pp_k$ were computed for a single subdomain with extension radius $r^\pp$ = $0.2r_c^\pp,\;0.6r_c^\pp$ and $r_c^\pp$. The results are shown in Figure~\ref{fig:sigma_r}, and the extended subdomains with different radii are visualised in Figure~\ref{fig:r_exts}. As expected, the singular values decay much faster for larger $r^\pp$. This comes at higher computational cost due to increase in the extended subdomain DOFs. At the same time, the faster decay of singular values leads to smaller $\mathrm{dim}(\tV_h)$.

\begin{figure}
\begin{minipage}{0.45\textwidth}
\centering
  \includegraphics[width=.8\linewidth]{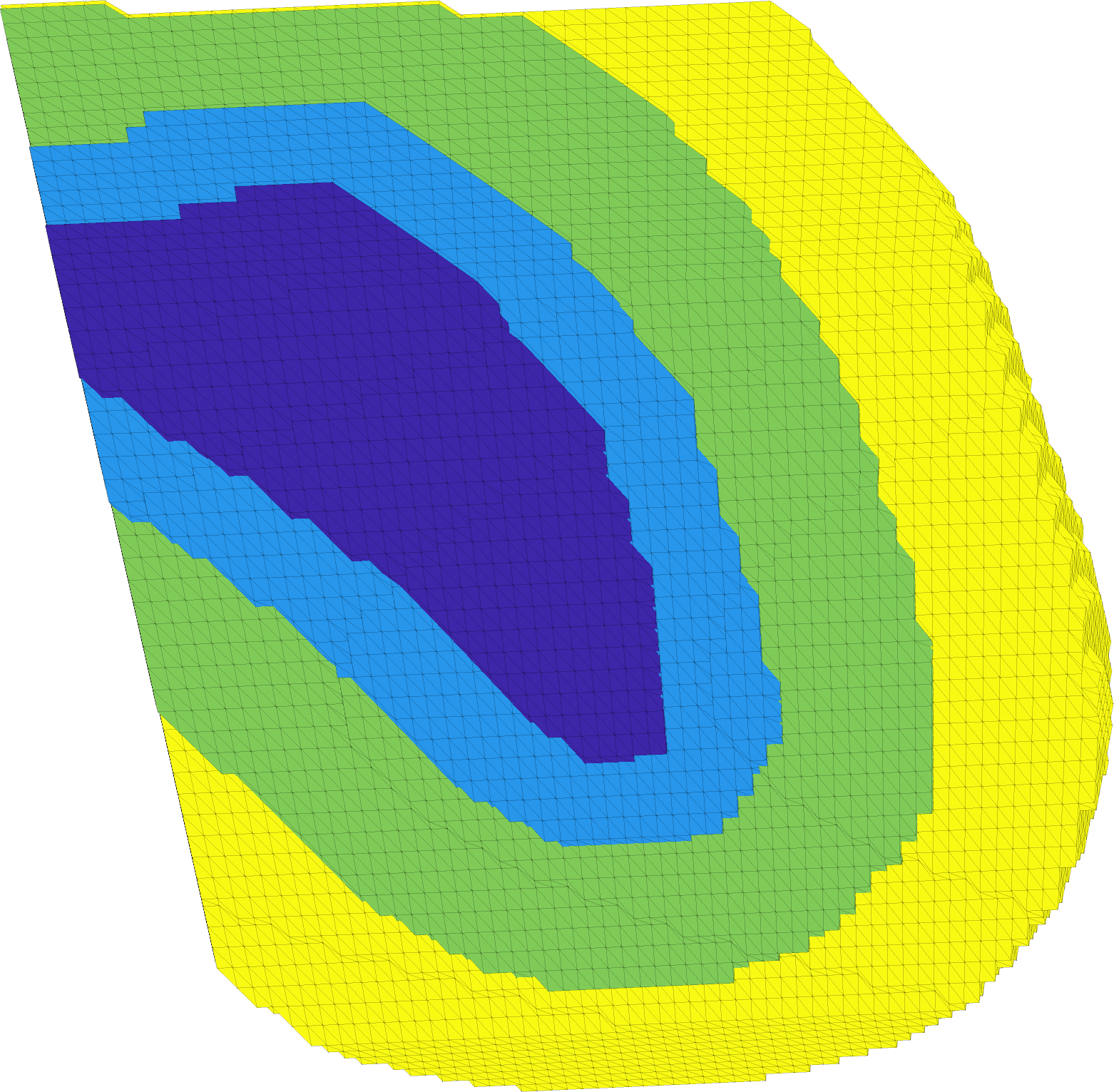}
  \caption{Examples of extended subdomains $\Uhp$ for extension radii $0.2r_c^\pp,0.6r_c^\pp$,  and $r_c^\pp$. The figure depicts a cross-section where $\Up$ is colored in dark blue.}
  \label{fig:r_exts}
\end{minipage}
\hfill
\begin{minipage}{0.45\textwidth}
  \centering
  \includegraphics[width=.8\linewidth]{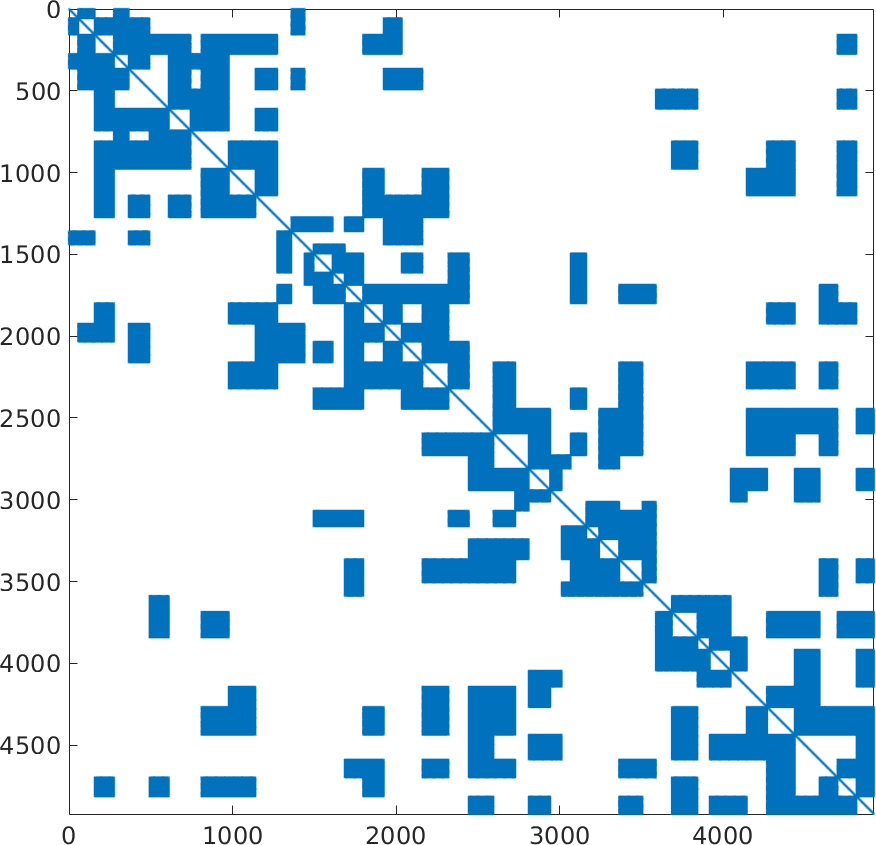}
  \caption{Sparsity pattern of the PU-CPI stiffness and mass matrices corresponding to \eqref{eq:red_eigen} posed on $\tV_h$ with $\mathrm{dim}(\mathcal{V}_h) = 195112$, $\mathrm{dim}(\tV_h) = 492$ and $M=44$.}
  \label{fig:nnz}
\end{minipage}
\end{figure}

\subsection{The effect of the cut-off tolerance of singular values}
\label{sec:svd-compression}
The computations were performed using three different mesh densities and several values of $tol$. The maximum relative
eigenvalue error and $\mathrm{dim}(\tV_h)$ are shown in Figure~\ref{fig:tol_error}. Additionally, relative error for each of the $200$ lowest eigenvalues are detailed in Figure~\ref{fig:mode_errors}. These results verify the linear relationship between $tol^2$ and the relative eigenvalue error predicted in Section~\ref{sec:algorithm}. In this examples, choosing $tol=1$ already produces relative eigenvalue error smaller than $1\%$.

\begin{table}
  \centering
  \caption{Relative fill-in is the ratio of the number of non--zeros in the stiffness matrices from spaces $\tV_h$ and $\mathcal{V}_h$. The last column is the average size of the dimension of the local method subspaces. Relative error is not given if the problem could not be solved using MATLAB \texttt{eigs} on a single workstation. The number of subdomains in computations, except for those with three largest $\mathrm{dim}(\mathcal{V}_h)$, was chosen so that each subdomain had about $5000$ vertices, see~Section~\ref{sec:varying-mesh-density}.  }
  \label{tab:ndof_error}
\begin{tabular}{c|c|c|c|c|c}
$\mathrm{dim}(\mathcal{V}_h)$ & $\mathrm{dim}(\widetilde{\mathcal{V}}_h)$ & $M$ & rel. fill-in \% & max rel error & $\mathrm{avg_p} \{ \dim \tV(\Up) \}$  \\
  \hline
\num{54872} & \num{2713} & 13 & 436.0 & \num{4.28e-05} & 209 \\
\num{110592} & \num{3777} & 25 & 291.9 & \num{1.90e-04} & 151 \\
\num{195112} & \num{4920} & 44 & 171.9 & \num{2.99e-04} & 112 \\
\num{314432} & \num{6048} & 69 & 121.1 & \num{3.73e-04} & 88 \\
\num{474552} & \num{7686} & 103 & 85.6 & \num{3.47e-04} & 75 \\
\num{681472} & \num{9791} & 146 & 76.9 & \num{3.81e-04} & 67 \\
\num{941192} & \num{12398} & 200 & 68.2 & - & 62 \\
\num{1259712} & \num{15587} & 267 & 60.8 & - & 58 \\
\num{1643032} & \num{19735} & 346 & 59.9 & - & 57 \\
\num{2097152} & \num{24276} & 440 & 56.3 & - & 55 \\
\num{2628072} & \num{30124} & 549 & 56.4 & - & 55 \\
\num{3241792} & \num{12820} & 150 & 26.7 & - & 85 \\
\num{5000211} & \num{19261} & 250 & 24.7 & - & 77 \\
\num{10360232} & \num{29124} & 308 & 22.8 & - & 95 
\end{tabular}
\end{table}
\begin{table}
  \caption{Second column:  Average time over $p$ of computing bases for local method subspaces $\tV_h(\Up)$ by workers. Third column:  Time required to partition the mesh by the master. Fourth column:  Time required to construct the $r$-extensions by the master. Remaining columns: Solution time for \eqref{eq:red_eigen} posed in PU-CPI method subspace with MATLAB's \texttt{eigs} ($\mathrm{t_{red}}$), total PU-CPI computational time ($\mathrm{t_{CPI}}$), and time required by direct FE-solution of \eqref{eq:cont_eigen} ($\mathrm{t_{FEM}}$). Cases where the problem could not be solved on a single workstation are marked with --. METIS was not used for the densest mesh, and all times are in seconds.}
  \label{tab:ndof_t}
  \centering
\begin{tabular}{c|c|c|c|c|c|c}
$\mathrm{dim}(\mathcal{V}_h)$ & avg. $\mathrm{t_{sub}}$ &  METIS & $r$-ext. & $\mathrm{t_{red}}$ & $\mathrm{t_{CPI}}$ & $\mathrm{t_{FEM}}$ \\
  \hline
\num{54872} & 32.3 & 1.6 & 6.5 & 13.3 & 105.1 & 53.9 \\
\num{110592} & 35.5 & 3.4 & 13.6 & 19.2 & 130.0 & 142.2 \\
\num{195112} & 48.8 & 6.2 & 31.7 & 25.2 & 177.2 & 307.5 \\
\num{314432} & 59.1 & 10.3 & 48.6 & 33.8 & 227.3 & 661.7 \\
\num{474552} & 63.8 & 16.7 & 78.3 & 43.5 & 303.7 & 1264.7 \\
\num{681472} & 68.0 & 24.1 & 123.8 & 56.1 & 429.7 & 1859.7 \\
\num{941192} & 71.2 & 34.7 & 189.7 & 79.4 & 610.8 & - \\
\num{1259712} & 75.5 & 47.2 & 263.8 & 107.0 & 772.7 & - \\
\num{1643032} & 75.8 & 65.0 & 357.6 & 151.6 & 1030.5 & - \\
\num{2097152} & 79.9 & 85.7 & 511.0 & 212.5 & 1397.2 & - \\
\num{2628072} & 79.6 & 107.3 & 671.3 & 302.4 & 1781.5 & - \\
\num{3241792} & 378.9 & 114.8 & 1562.0 & 165.5 & 4509.7 & - \\
\num{5000211} & 385.3 & 327.4 & 1680.4 & 306.1 & 6312.9 & - \\
\num{10360232} & 344.4 & - & 1365.4 & 690.3 & 6525.7 & - 
\end{tabular}
\end{table}

\begin{figure}
  \centering
  \includegraphics[width=.45\textwidth]{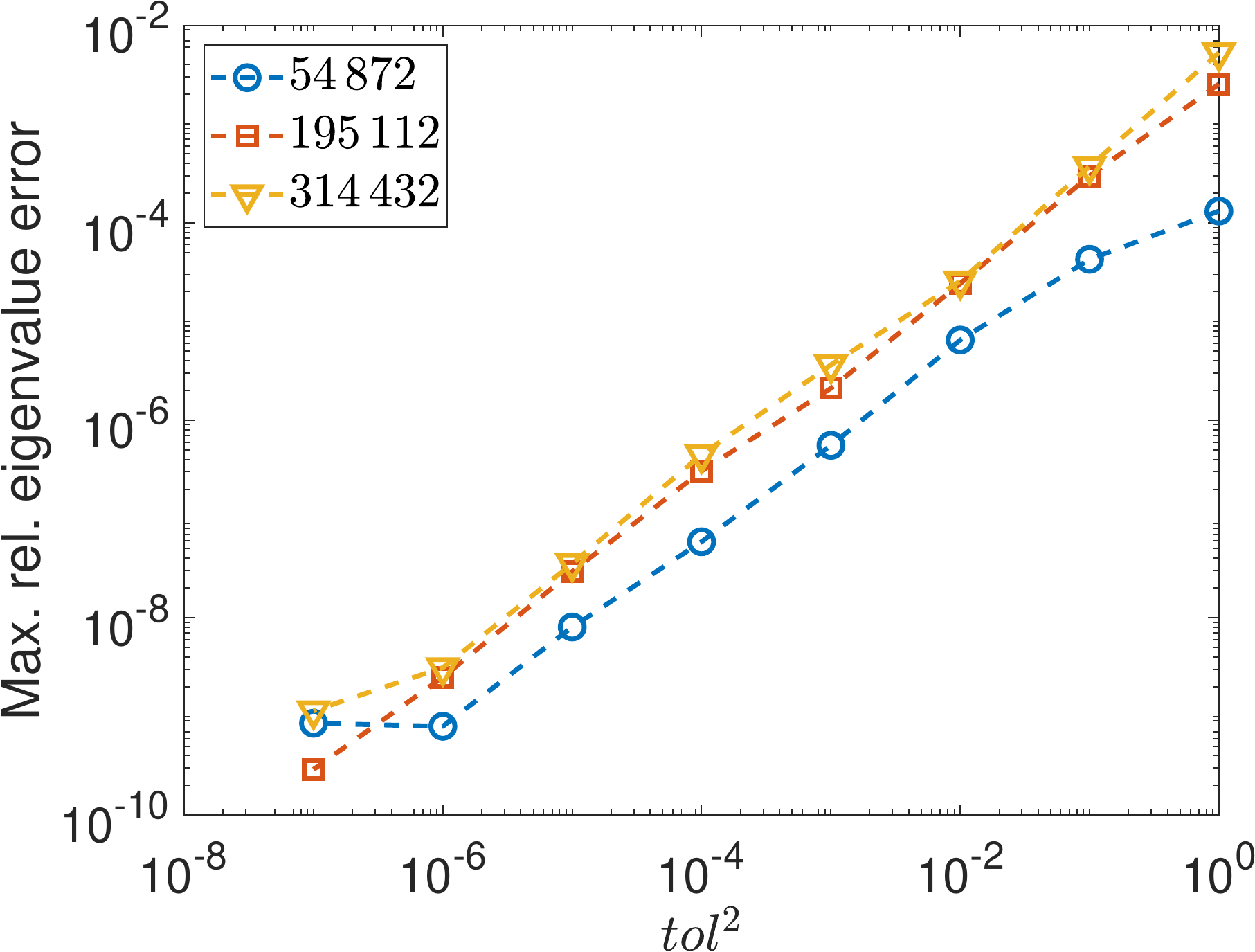} \hfill
  \includegraphics[width=.45\textwidth]{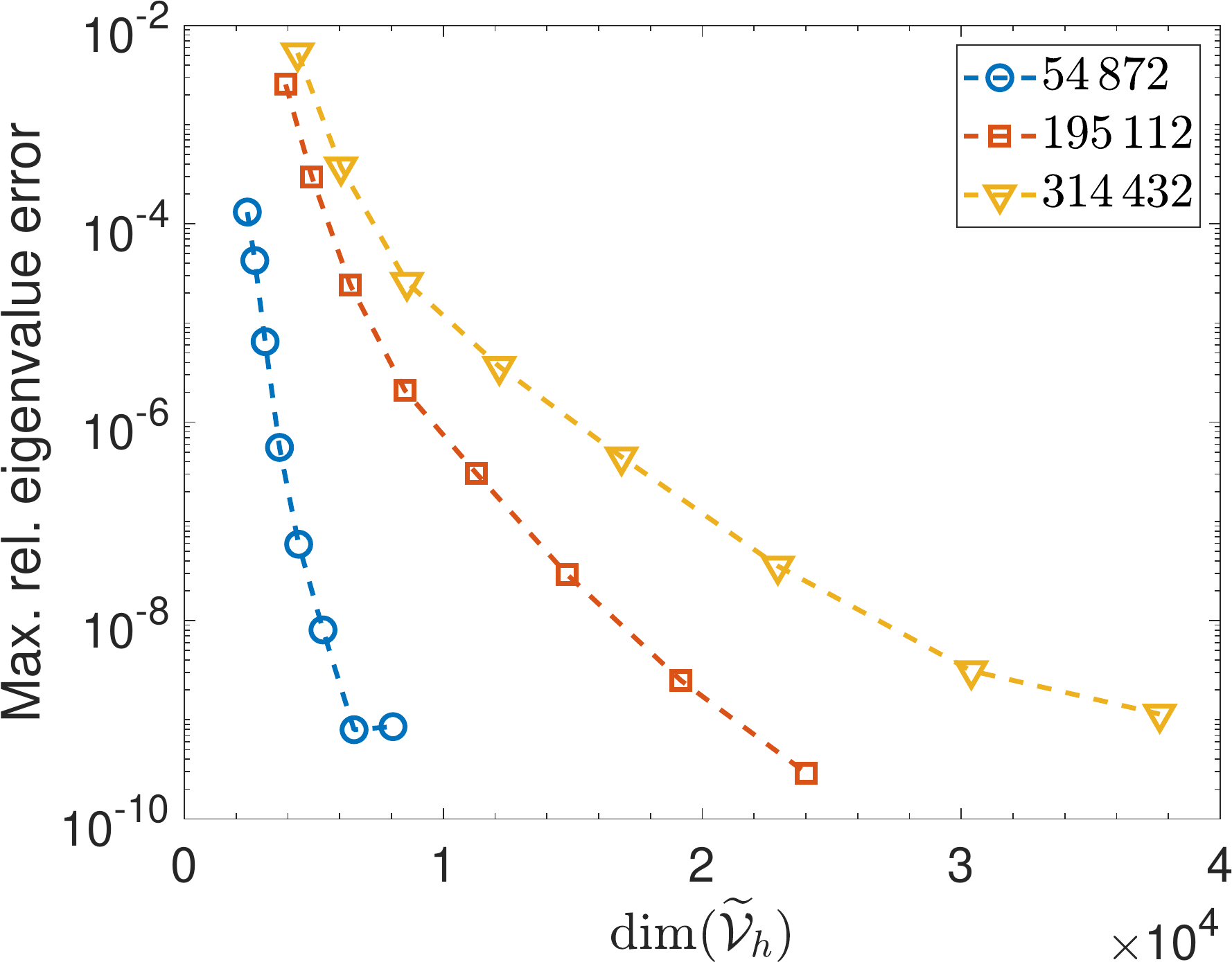}
  \caption{Maximum relative eigenvalue error using three mesh densities. Left panel: Given as a function of the cut-off tolerance for singular values
    $\mathrm{tol}$. Right panel: Given as a function of $\mathrm{dim}(\tV_h)$  in the same sample points.  }
  \label{fig:tol_error}
\end{figure}

\begin{figure}
  \centering
  \includegraphics[width=.45\textwidth]{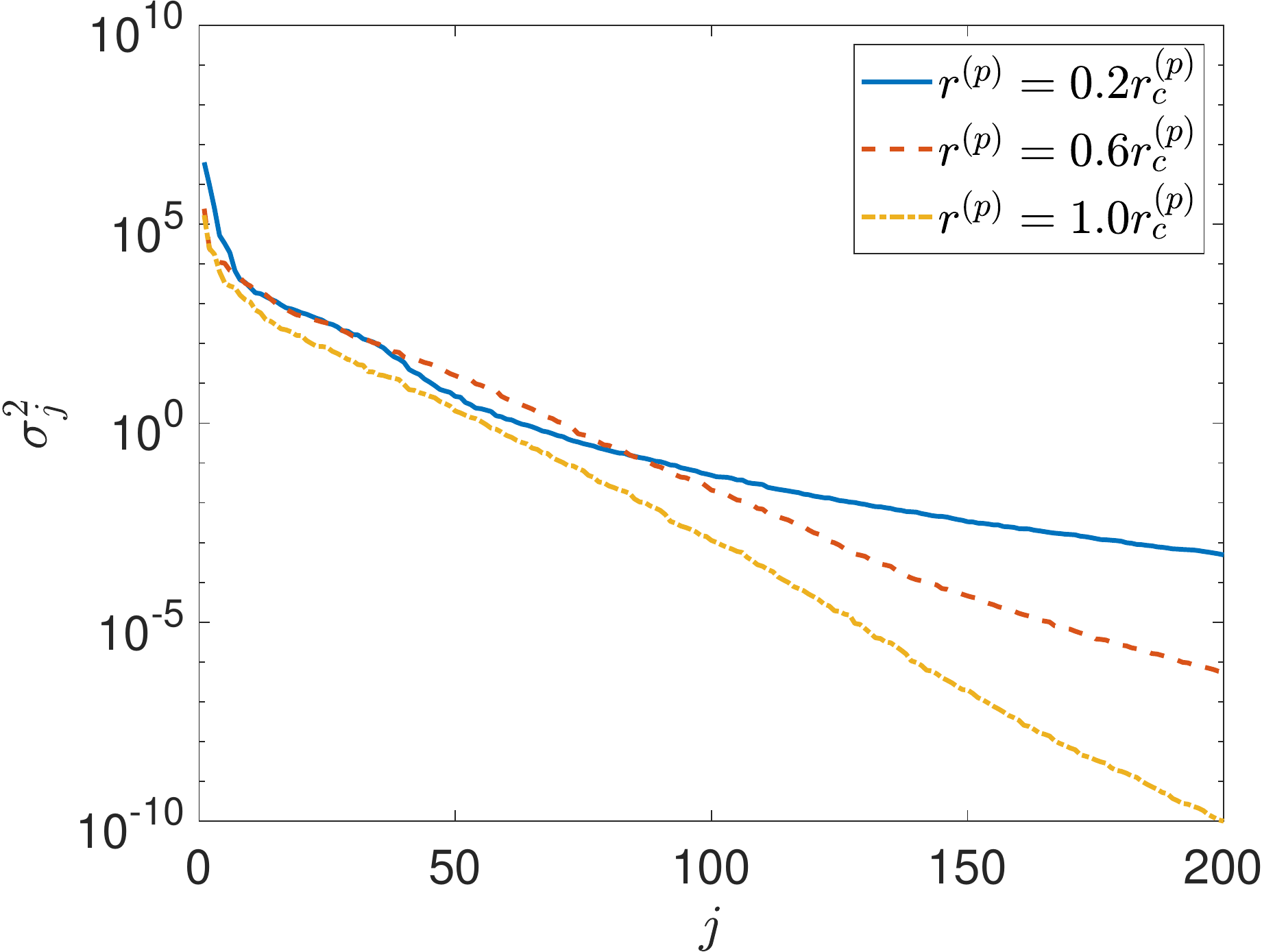} \hfill
  \includegraphics[width=.45\textwidth]{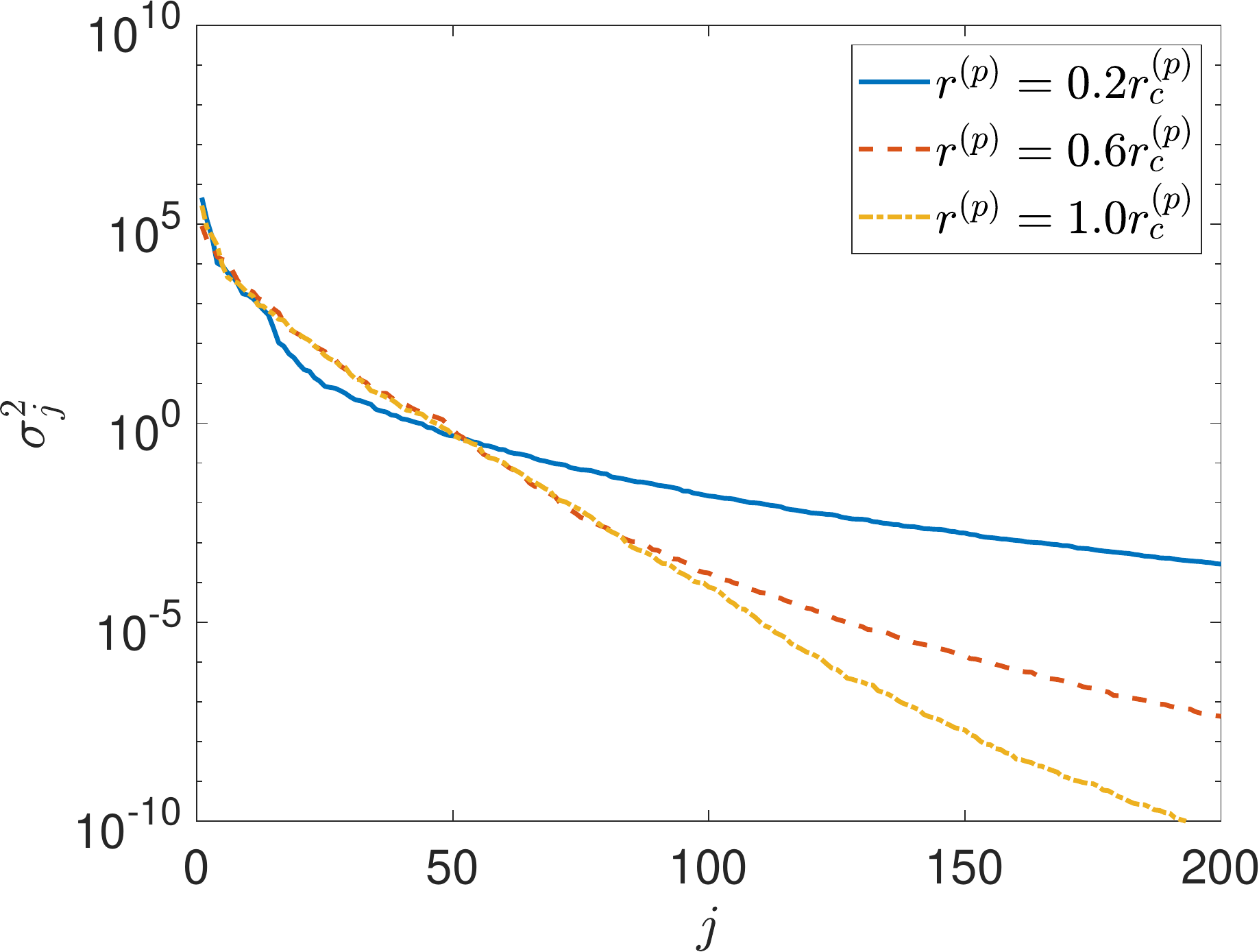}
  \caption{Effect of the extension radius on the decay of singular
    values $\{ \sigma_j \}_{j=1}^{200}$ for the cube with $\num{54872}$ (left panel) and $\num{195112}$
    (right panel) DOFs. In these experiments, extended subdomain DOFs range
    between $\num{10610}-\num{32423}$ and $\num{12644}-\num{49659}$,
    respectively.}
  \label{fig:sigma_r}
\end{figure}

\begin{figure}
  \centering
  \includegraphics[width=.45\textwidth]{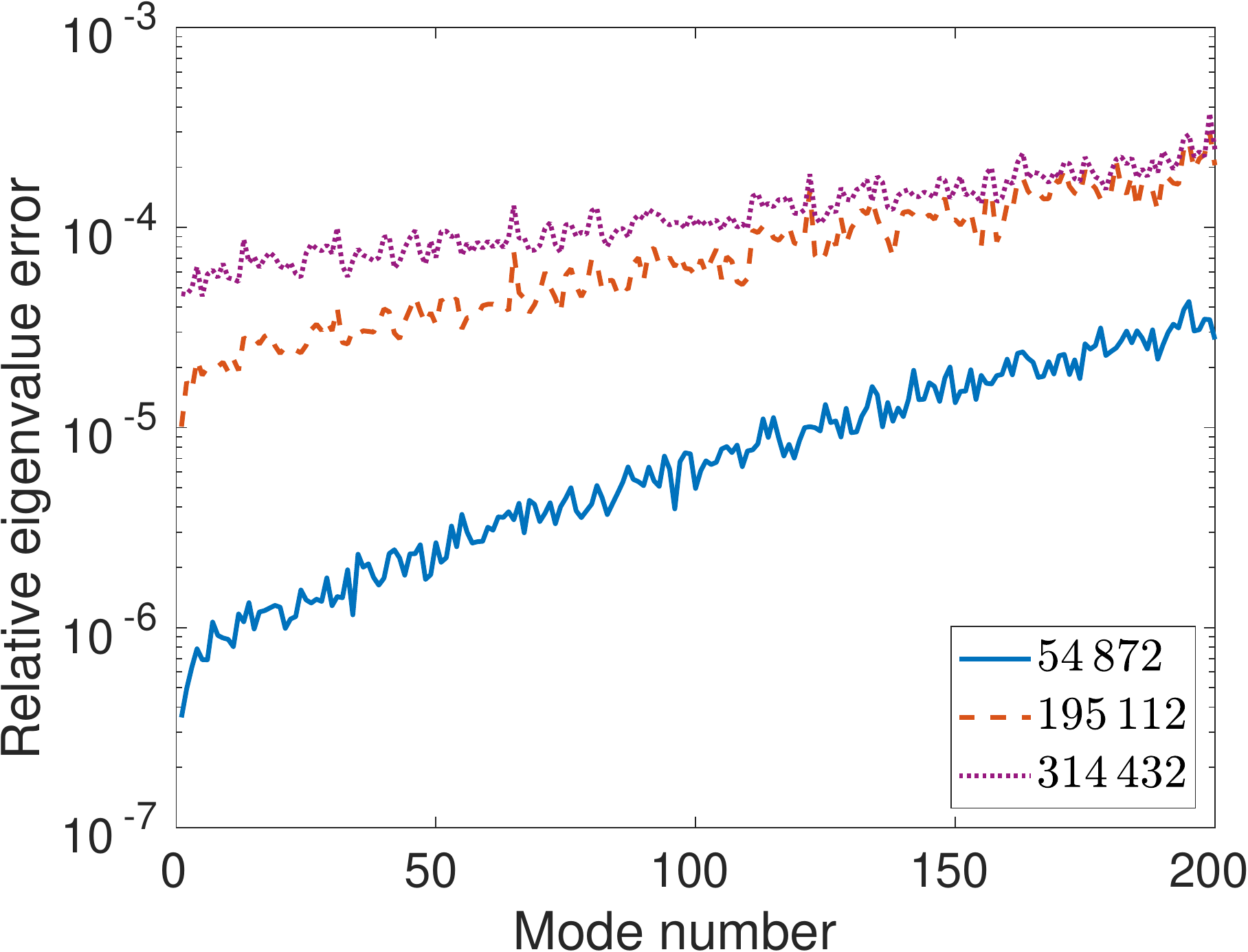} \hfill
  \includegraphics[width=.45\textwidth]{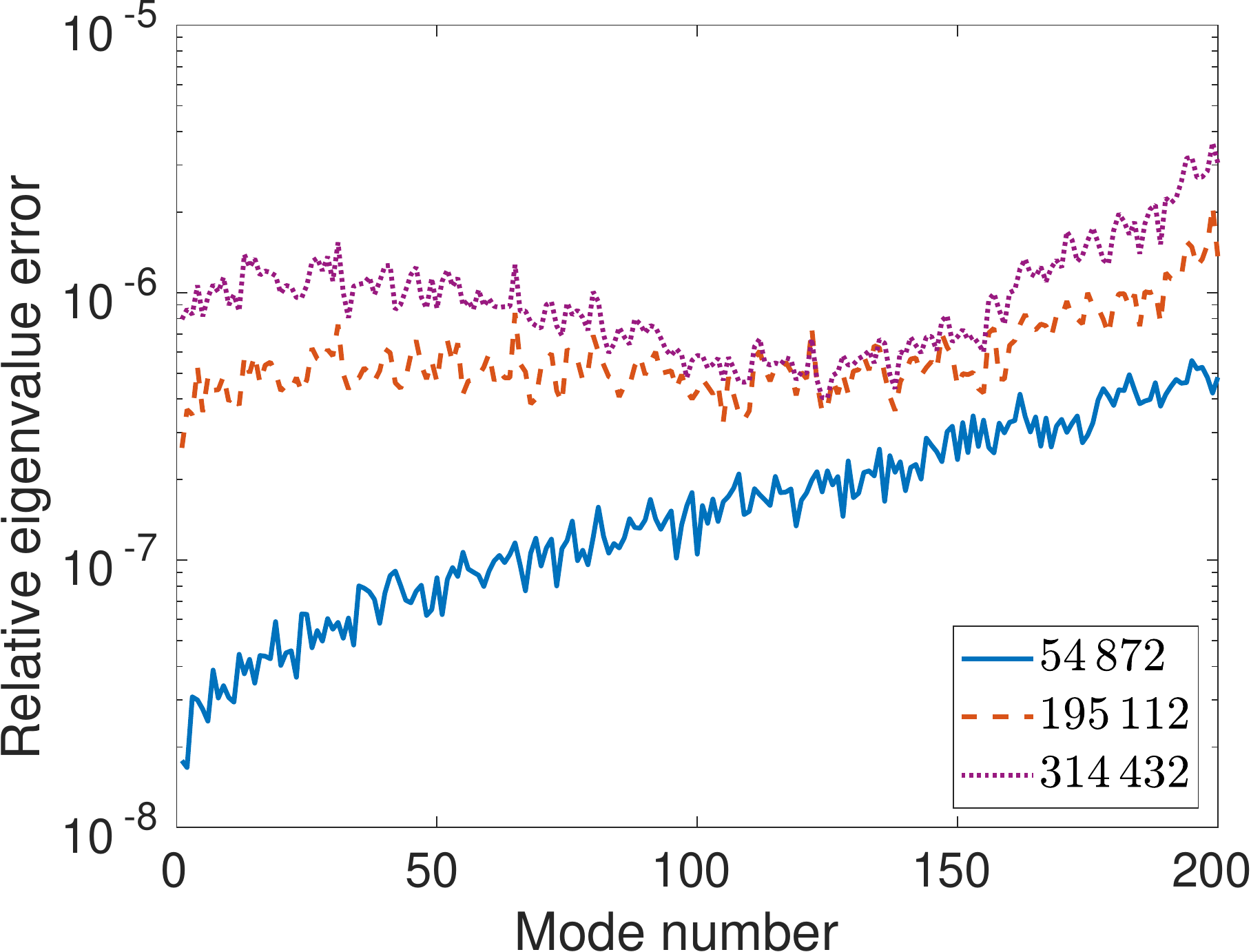}
  \caption{Relative eigenvalue errors for the $200$ lowest modes for three mesh
    densities. The cut-off tolerance for singular values $tol^2 = 0.1$ (left panel) and
    $tol^2 = 0.001$ (right panel).}
  \label{fig:mode_errors}
\end{figure}


\section{Conclusions}

PU-CPI method for the approximate solution of eigenvalues in $(0,\Lambda)$ of the Dirichlet Laplacian on domain $\Omega$ is proposed. PU-CPI is a Ritz method where the method subspace $\tV$ is constructed from the local method subspaces $\{ \tV(\Up)\}_p$ for $\Up \subset \Omega$ as stated in \eqref{eq:PUMdef}. Since the local subspaces are independent of each other, PU-CPI can be used in distributed computing environments where communication is costly. Failed distributed tasks can be restarted, making the implementation of PU-CPI very robust. 

Let $(u,\lambda)$ be solution of~\eqref{eq:cont_eigen} for $\lambda \in (0,\Lambda)$. According to Proposition~\ref{prop:eigen_proj_estimate} and \eqref{eq:pum_error}, the local method subspaces should be designed to approximate $u|_\Up$. Local information on $u|_\Up$ is obtained in terms of the operator-valued function $Z_U$ in Lemma~\ref{prop_extension}. Since $Z_U$ is compact operator-valued by Lemma~\ref{lemma:compact}, its values can be efficiently low-rank approximated. 

The local method subspace for the single subdomain $U \equiv \Up$ is designed to approximate range of $Z_U$ in the sense of~\eqref{eq:WUerror}. This approximation makes use of interpolation, linearisation, and low-rank approximation as explained in Section~\ref{sec:complementing_subspace}. The local approximation error is estimated in Theorems~\ref{cor:Ziperror} and~\ref{thm:trunc_err_est}. Theorem~\ref{thm:main1} combines these estimates to bound the global relative eigenvalue error.

An example of low-rank approximation is given for the first-order FEM in Theorem~\ref{thm:main2}. The key ingredient is Lemma~\ref{lemma:H1_2norm} and Remark~\ref{remark:Sinv} that allow numerical treatment of a required boundary trace norm. A basis for each local method subspace is obtained from eigenvectors of the corresponding $\mathsf{C} \mathsf{C}^T$ in~\eqref{eq:ckdef}. The dimension of $\mathsf{C} \mathsf{C}^T$ is independent of parameters $N$ and $\eta$.

Finally, numerical examples validating the theoretical results and demonstrating the potential of PU-CPI are given in Section~\ref{sec:num_ex}. The authors could use inexpensive networked workstations to solve an eigenvalue problem ten times as large as straightforwardly solvable on a single workstation. In contrary to using a supercomputer, such networked workstations are widely available.

The dimension of PU-CPI method subspace $\tV$ is related to the number of singular values of each $\mathsf C \mathsf C^T$ larger than given $tol>0$. Nothing in our theoretical work indicates how the fast singular values decay or estimate the dimension of $\tV$. The numerical results in Fig.~\ref{fig:sigma_r} indicate exponential decay with a rate dependent on the extension radius $r$, which we believe to be a generic property of similar elliptic problems. All this remains a topic of further research.  

Acoustic eigenvalues problem, for example, benefit from treatment of more general boundary conditions. The authors have implemented PU-CPI for mixed homogeneous Dirichlet and Neumann boundary conditions, and the error analysis extends to this case.

\section{Acknowledgements}
The authors are grateful for the comments of the reviewers.

\bibliographystyle{siamplain}
\bibliography{master_new}

\end{document}